\documentclass[a4paper,12pt,reqno]{amsart}

 \newtheorem{thm}{Theorem}[section]
 \newtheorem{cor}[thm]{Corollary}
 \newtheorem{lem}[thm]{Lemma}
 \newtheorem{prop}[thm]{Proposition}
 \newtheorem{cond}[thm]{Condition}
 \theoremstyle{definition}
 \newtheorem{defn}[thm]{Definition}
 \theoremstyle{remark}
 
 \numberwithin{equation}{section}

 \newcommand{\R}{{\mathbb{R}}}
 \newcommand{\N}{{\mathbb{N}}}
 
\newcommand{\calA}{\mathcal{A}_{\alpha}}

\usepackage[margin=30mm]{geometry}
\usepackage{graphicx}

\usepackage{amsmath}
\usepackage[hidelinks]{hyperref}
\usepackage{mathtools}
\usepackage{dsfont}

\title[Intrinsic ultracontractivity for Schr{\"o}dinger semigroups]
 {Intrinsic ultracontractivity for Schr{\"o}dinger semigroups based on cylindrical fractional Laplacian on the plane}

\author{Tadeusz Kulczycki, Kinga Sztonyk}

\thanks{T. Kulczycki and K. Sztonyk were supported in part by the National Science Centre, Poland, grant no. 2019/35/B/ST1/01633}

\begin{document}

\abstract
We study Schr{\"o}dinger operators on $\R^2$
$$
H = \left(-\frac{\partial^2}{\partial x_1^2}\right)^{\alpha/2} + \left(-\frac{\partial^2}{\partial x_2^2}\right)^{\alpha/2} + V,
$$
for $\alpha \in (0,2)$ and some sufficiently regular, radial, confining potentials $V$. 
We obtain necessary and sufficient conditions on intrinsic ultracontractivity for semigroups $\{e^{-tH}: \, t \ge 0\}$. We also get sharp estimates of first eigenfunctions of $H$.
\endabstract

\makeatletter
\def\@setkeywords{\@ifempty{\keywords}{}
  {\footnotesize\textbf{Keywords: }\@keywords \\}}
\makeatother


\maketitle

\section{Introduction}
The following anisotropic, singular L{\'e}vy operator in $\R^{3N}$
$$
L^{(rel)} = -\sum_{k = 1}^N 
\sqrt{-\frac{\partial^2}{\partial x_{3k-2}^2} -\frac{\partial^2}{\partial x_{3k-1}^2} -\frac{\partial^2}{\partial x_{3k}^2} + m^2} + Nm,
$$
where $m > 0$, $N \in \N$, plays an important role in some models of relativistic quantum mechanics. For example, in Lieb and Seiringer's book \cite[see formula (3.2.6)]{LS2007} it is a kinetic term in the relativistic Schr{\"o}dinger Hamiltonian for $N$ electrons. It appears also in the relativistic spinless Salpeter equation for a two-body system, see e.g. \cite[formula (12)]{S2019}, and in the model of a relativistic gravitating system of $N$ neutral particles, see e.g. \cite[formula (3.2.11)]{LS2007}.

In our paper we study Schr{\"o}dinger operators based on the following anisotropic, singular L{\'e}vy operator in $\R^{2}$, which is called the cylindrical fractional Laplacian
$$
L^{(\alpha)} = -\left(-\frac{\partial^2}{\partial x_1^2}\right)^{\alpha/2} - \left(-\frac{\partial^2}{\partial x_2^2}\right)^{\alpha/2},
$$
where $\alpha \in (0,2)$. For $\alpha = 1$ this operator may be viewed as a simplified version of $L^{(rel)}$. We believe that methods from our paper will be useful in investigating Schr{\"o}dinger operators based on $L^{(rel)}$.

Schr{\"o}dinger operators considered in this paper have the following form 
$$
H f(x) = - L^{(\alpha)} f(x) + V(x) f(x), \quad x \in \R^2,
$$ 
where potentials $V$ satisfy the following assumptions.

\textbf{Assumptions (A).} Let $V: \R^2 \to [0,\infty)$ be a radial function, which has a profile function $q: [0,\infty) \to [0,\infty)$ (that is $q(|x|) = V(x)$, for any $x \in \R^2$) satisfying the following conditions.

a) $\displaystyle \lim_{x \to \infty} q(x) = \infty$,

b) $q$ is nondecreasing on $[0,\infty)$,

c) $q$ is continuous on $[0,\infty)$,

d) there exists $C_0 \ge 1$ such that for every $x \in [0,\infty)$
$$
q(x + 1) \le C_0 (q(x) + 1).
$$

Typical examples of potentials $V$ satisfying assumptions (A) are $V(x) = |x|^{\beta}$, $V(x) = \exp(\beta |x|)$ for $\beta > 0$. On the other hand, functions $V(x) = \exp(|x|^{\beta})$ for $\beta > 1$ fail to satisfy condition d).

The Schr{\"o}dinger operator $H$ generates a semigroup of symmetric operators $\{e^{-tH}: \, t \ge 0\}$ on $L^2(\R^2)$, such that $e^{-tH}: \, L^2(\R^2) \to L^{\infty}(\R^2)$ for $t > 0$. It is well known that $T_t := e^{-tH}$ is an integral operator, whose kernel is given by the function $u_t(x,y)$, i.e. $T_t f(x) = \int_{\R^2} u_t(x,y) f(y) \, dy$, $f \in L^2(\R^2)$, $x \in \R^2$, $t > 0$ \cite[Chapter 2.B]{DC2000}. For any $t > 0$ the kernel $u_t(x,y)$ is continuous, bounded and positive on $\R^2 \times \R^2$ (the proof of this property is similar to the proof of Lemma 3.1 in \cite{KS2006}). Each $T_t$ for $t > 0$ is a compact operator in $L^2(\R^2)$ (see e.g. \cite{KK2010}, Lemmas 1 and 9 for a general argument) and the spectrum of $H$ consists of a sequence of eigenvalues $\{\lambda_n\}_{n = 1}^{\infty}$ satisfying
$$
0 < \lambda_1 < \lambda_2 \le \lambda_3 \le \ldots, \quad \lim_{n \to \infty} \lambda_n = \infty.
$$
The corresponding eigenfunctions $\{\varphi_n\}_{n = 1}^{\infty}$ may be chosen so that they form an orthonormal basis in $L^2(\R^2)$, all $\varphi_n$ are continuous and bounded on $\R^2$ and the first eigenfunction $\varphi_1$ is strictly positive on $\R^2$. Clearly, we have $T_t \varphi_n(x) = e^{-\lambda_n t} \varphi_n(x)$ for any $n \in \N$, $t > 0$, $x \in \R^2$.

It is well known that the Schr{\"o}dinger semigroup $\{T_t: \, t \ge 0\}$ has the following probabilistic Feynman–Kac representation. For any $f \in L^2(\R^2)$, $t > 0$ the following equality holds (in the sense of $L^2(\R^2)$)
\begin{equation}
\label{FeynmanKac}
T_t f(x) = E^x\left(\exp\left(-\int_0^t V(X_s) \, ds\right) f(X_t)\right),
\end{equation}
which allows us to use probabilistic methods. 

The most important result of this paper deals with intrinsic ultracontractivity (IU) for the semigroup $\{T_t: \, t \ge 0\}$ and estimates of the first eigenfunction $\varphi_1$. IU was introduced by E. B. Davies and B. Simon in \cite{DS1984}. The semigroup $\{T_t:\,  t \ge 0\}$ is called intrinsically ultracontractive if and only if for any $t > 0$ there exist positive constants $c_{1,t}, c_{2,t}$ such that for all $x,y \in \R^2$ we have
$$
c_{1,t} \varphi_1(x) \varphi_1(y) \le u_t(x,y) \le c_{2,t} \varphi_1(x) \varphi_1(y).
$$
The main result of this paper is the following theorem (recall that $q$ is a profile function for $V$).

\begin{thm}\label{main}
	Assume that $V$ satisfies assumptions (A). Then $\{T_t: \, t \ge 0\}$  is intrinsically ultracontractive if and only if
	\begin{equation}\label{warunek_logarytmu}
	\lim_{x \to \infty} \frac{q(x)}{\ln{x}} = \infty.
	\end{equation}
	Moreover, if \eqref{warunek_logarytmu} holds, then there exist constants $c_1 = c_1(\alpha, q) > 0$ and $c_2 = c_2(\alpha, q) > 0$, such that for every $x = (x_1,x_2) \in \mathbb{R}^2$
	\begin{equation*}
	\begin{multlined}[t] \frac{c_1}{(|x_1|+1)^{1+\alpha} (|x_2|+1)^{1+\alpha} (q(\min(|x_1|, |x_2|)) + 1) (q(|x|) + 1)} \leq \varphi_1 (x) \\
	\leq \frac{c_2}{(|x_1| + 1)^{1+\alpha} (|x_2| + 1)^{1+\alpha} (q(\min(|x_1|, |x_2|)) + 1) (q(|x|) + 1)}. \end{multlined}
	\end{equation*}
\end{thm}

IU has been introduced in \cite{DS1984} for very general semigroups. Important examples of such semigroups are those generated by Dirichlet Laplacians on domains in $\R^d$ and by Schr{\"o}dinger operators $H = -\Delta + V$ both on $\R^d$ as well as on domains of $\R^d$ (with Dirichlet boundary conditions). IU for such classical semigroups has been studied in e.g. \cite{DS1984, D1989, B1991, D1991}. In the recent years IU and ground state estimates have also been studied for Schr{\"o}dinger semigroups generated by $H = -L +V$ on $\R^d$, where $V$ is a confining potential in $\R^d$ and $L$ is a generator of a symmetric L{\'e}vy process in $\R^d$, which L{\'e}vy measure has a density (with respect to the Lebesgue measure on $\R^d$) satisfying some regularity conditions, see e.g. \cite{KL2015, CW2016, KS2020}.

The cylindrical fractional Laplacian $L^{(\alpha)}$, which we consider in our paper, is the generator of the process $X = (X^{(1)},X^{(2)})$ in $\R^2$, which components $X^{(1)}$, $X^{(2)}$ are independent, one-dimensional, symmetric, standard $\alpha$-stable processes. Note that $L^{(\alpha)}$ is a singular L{\'e}vy operator, because the L{\'e}vy measure of the corresponding process $X$ is singular with respect to the Lebesgue measure on $\R^2$ (it is supported by coordinate axes). The study of Schr{\"o}dinger operators based on such singular operator requires new methods. In particular in Section 4 we use some new probabilistic techniques in order to estimate $T_t \mathds{1}_{\R^2}(x)$.

The main motivation of our research comes from an attempt to thoroughly analyze Schr{\"o}dinger operators $H = -L^{(rel)} + V$, which appear in some models of relativistic physics. However, another motivation comes from mathematics. We are interested in understanding Schr{\"o}dinger operators $H = -L + V$, when $L$ is a singular, nonlocal L{\'e}vy operator, that is $L$ is a generator of a jump L{\'e}vy process with a singular L{\'e}vy measure. Typical examples of such processes are $X = (X^{(1)}, \ldots, X^{(d)})$, which components $X^{(1)}, \ldots, X^{(d)}$ are independent, one-dimensional jump processes. Such processes $X$ appear naturally in many models. In recent years singular, nonlocal L{\'e}vy operators have been intensively studied both by probabilists and specialists in PDEs, see e.g. \cite{BKS2017, BS2018, C2019, CK2020, CHZ2023, KKK2022, KW2022, KKS2021, ROV2016}.
An interesting open problem is the following.

\vskip 5pt

{\bf Open question.} What are the first eigenfunction estimates and conditions for IU for Schr{\"o}dinger semigroups corresponding to operators $H = -L +V$, when $L$ is a general symmetric stable operator in $\R^d$ (not necessarily isotropic)?

\vskip 5pt

Such general symmetric stable operators have been studied in e.g. \cite{ROS2016, Dz1991, GH1993}.

In our paper we use some probabilistic ideas from \cite{KS2006}, however, there are essential differences between our paper and \cite{KS2006} due to singularity of the L{\'e}vy measure of the process $X$. We also borrow some arguments from \cite{BK2003} (see Lemma 4.5 in \cite{BK2003}). 

The techniques used in the paper are quite complicated therefore we concentrate on the case of dimension $d = 2$. Our goal is to present the idea, not to achieve maximal generality with respect to dimension. We do not want the technicalities to obscure the approach we propose.

The paper is organized as follows. In Section \ref{Preliminaries} we introduce notation and collect known facts needed in the sequel. In Section 3 we obtain lower bound estimates of $T_t \mathds{1}_{D}(x)$. The most difficult and important is Section 4, where we prove an upper bound of $T_t \mathds{1}_{\R^2}(x)$. In Section 5 we use the results obtained in two previous sections to prove the main result of our paper -- Theorem \ref{main}.

\section{Preliminaries}
\label{Preliminaries}
In the whole paper we assume that $\alpha \in (0,2)$ and a potential $V$ satisfies Assumptions (A). We write $c = c(a, b, \dots)$ to indicate the dependence of a constant $c$ on parameters. Many constants depend on $\alpha$ and the potential $V$, so we usually omit dependence on $\alpha$ and $V$ in the notation. In this paper, the constants written with capital letters ($\calA$, $C_0$, $C_1$, $C_2$, \dots) do not change their value. On the other hand, constants written with lowercase letters ($c$, $c_1$, $c_2$, \dots) may change their value from one use to the next. We assume that all constants are positive. We denote $\mathbb{N} = \{1, 2, 3, \ldots\}$.

Let $X = (X^{(1)},X^{(2)})$ be a cylindrical $\alpha$-stable process in $\R^2$. Its components $X^{(1)}$, $X^{(2)}$ are independent, one-dimensional, symmetric, standard $\alpha$-stable processes. The characteristic function of $X_t^{(1)}$ is given by $E^{0} e^{i X_t^{(1)} \xi} = e^{-t |\xi|^{\alpha}}$, where $t > 0$, $\xi \in \R$. The transition density of $X^{(1)}$ is denoted by $\tilde{p}(t,x,y)$, $t > 0$, $x, y \in \R$. The transition density of $X$ is denoted by $p(t,x,y)$, $t > 0$, $x, y \in \R^2$. Clearly, we have $p(t,x,y) = \tilde{p}(t,x_1,y_1) \tilde{p}(t,x_2,y_2)$ for any $t > 0$, $x = (x_1,x_2) \in \R^2$, $y = (y_1,y_2) \in \R^2$. By $\tau_D$ we denote the first exit time for an open set $D \subset \R^2$, i.e. $\tau_D = \inf\{t \ge 0: \, X_t \notin D\}$. By $p_D(t,x,y)$ we denote the density of the process $X$ started at $x$ and killed on exiting an open set $D \subset \R^2$
$$
p_D(t,x,y) = p(t,x,y) - E^x(\tau_D \le t, p(t - \tau_D, X(\tau_D),y)),
$$
for $x ,y \in D$. Let $a, b, c, d \in \R$ with $a < b$, $c < d$. For $D = (a,b) \times (c,d)$, $t > 0$, $x = (x_1,x_2) \in \R^2$, $y = (y_1,y_2) \in \R^2$ we have
\begin{equation}
\label{rectangle}
p_D(t,x,y) = \tilde{p}_{(a,b)}(t,x_1,y_1)  \tilde{p}_{(c,d)}(t,x_2,y_2),
\end{equation}
where $\tilde{p}_{U}(t,\cdot,\cdot)$ is the transition density of the process $X^{(1)}$ killed on exiting an open set $U \subset \R$.

It is well known that the L{\'e}vy measure of the process $X$ is given by
$$
\mu(dz_1 \, dz_2) = \delta_0(dz_1) \times \frac{\calA}{|z_2|^{1+\alpha}} dz_2 + \frac{\calA}{|z_1|^{1+\alpha}} dz_1 \times \delta_0(dz_2),
$$
where $\calA = \alpha 2^{\alpha - 1} \Gamma((1+\alpha)/2) \pi^{-1/2} (\Gamma(1-\alpha/2))^{-1}$ and $\delta_x$ denotes the Dirac measure on $\R$ for $x \in \R$.

For any $y = (y_1,y_2) \in \R^2$ we denote 
$$
\mu_y(dz_1 \, dz_2) = \delta_{y_1}(dz_1) \times \frac{\calA}{|z_2 - y_2|^{1+\alpha}} dz_2 + \frac{\calA}{|z_1 - y_1|^{1+\alpha}} dz_1 \times \delta_{y_2}(dz_2).
$$
Now we state a generalization of the Ikeda-Watanabe formula (expressing the distribution of $(\tau_D, X(\tau_D))$) for the process $X$ (cf. \cite{IW1962}). 

\begin{prop}
Assume that $D$ is open, nonempty, bounded subset of $\R^2$ and $A$ is a Borel set such that $A \subset D^c \setminus \partial D$. Then for any $x \in D$, $0 \le t_1 < t_2 \le \infty$ we have
\begin{equation}
\label{IW_formula}
P^x(X(\tau_D) \in A, t_1 < \tau_D < t_2) = \int_D \int_{t_1}^{t_2} p_D(s,x,y) \, ds  \mu_y(A) \, dy.
\end{equation}
\end{prop}
The proof is almost the same as the proof of \cite[Proposition 2.5]{KS2006} and it is omitted. Let $a, b, c, d \in \R$ with $a < b$, $c < d$. When $D = (a,b) \times (c,d)$ and $x = (x_1,x_2) \in D$ then 
$$
P^x(X(\tau_D) \in \partial D) \le P^{x_1}(X^{(1)}(\tilde{\tau}_{(a,b)}^{(1)}) \in \{a,b\}) + P^{x_2}(X^{(2)}(\tilde{\tau}_{(c,d)}^{(2)}) \in \{c,d\}) = 0, 
$$
where 
$\tilde{\tau}_{(a,b)}^{(1)} = \inf\{t \ge 0: \, X_t^{(1)} \notin (a,b)\}$ and 
$\tilde{\tau}_{(c,d)}^{(2)} = \inf\{t \ge 0: \, X_t^{(2)} \notin (c,d)\}$.

We will study IU of the semigroup $\{T_t: \, t \ge 0\}$ checking the following conditions.

\begin{cond}\label{warunek_IU}
	There exists an open, bounded and nonempty set $D$, such that for every $t>0$ there exists a constant $c(t, D) > 0$, such that for every $x \in \mathbb{R}^2$
	\begin{equation*}
	T_t \mathds{1}_{\mathbb{R}^2} (x) \leq c(t, D) T_t \mathds{1}_D (x).
	\end{equation*}
\end{cond}

\begin{cond}\label{warunek_IU2}
	For any open, bounded and nonempty set $D$ and for any $t>0$ there is a constant $c(t,D)>0$ such that for any $x\in\mathbb{R}^2$
	\begin{equation*}
	T_t \mathds{1}_{B(x,1)}(x)\leq c(t,D) T_t \mathds{1}_D(x).
	\end{equation*}
\end{cond}

It is known (see e.g. arguments from \cite[Section 3]{KS2006}) that Condition \ref{warunek_IU} implies IU of $\{T_t: \, t \ge 0\}$ and that IU of $\{T_t: \, t \ge 0\}$ implies Condition \ref{warunek_IU2}. 

Now we present some estimates, which will be needed in the sequel. 

\begin{lem}
\label{series}
	For any $r > 0$ we have
	\begin{eqnarray*}
	\sum_{n=1}^{\infty} \frac{e^{-\frac{r}{n}}}{n(n+1)} &\geq& \frac{e^{-1}}{r+1},\\
	\sum_{n = 1}^{\infty} \frac{e^{-\frac{r}{n+1}}}{n(n+1)} &\leq& \frac{5}{r}.
	\end{eqnarray*}
\end{lem}
The above result follows from Lemmas 5.2 and 5.6 from \cite{KS2006}.

\begin{lem}\label{auxiliary_series}
	
	(i) Let $n, k \in \mathbb{Z}$ and $\mathcal{P}' = \mathbb{Z} \setminus \{k, n\}$. Then
	\begin{equation*}
	\sum_{p \in \mathcal{P}'} \frac{1}{|n-p|^{1+\alpha}} \frac{1}{|p-k|^{1+\alpha}} \leq \frac{C_1}{(|n-k|+1)^{1+\alpha}}.
	\end{equation*}
	
	(ii) Let $n, k \in \mathbb{Z}$ and $\mathcal{P}'' = \mathbb{Z} \setminus \{k\}$. Then
	\begin{equation*}
	\sum_{p \in \mathcal{P}''} \frac{1}{(|n-p|+1)^{1+\alpha}}   \frac{1}{|p-k|^{1+\alpha}} \leq \frac{C_2}{(|n-k| + 1)^{1+\alpha}}.
	\end{equation*}
\end{lem}
The proofs of the above estimates are elementary and are omitted. We may assume that $C_1, C_2 \ge 1$.

Note that if $n \in \N$ and $a \ge n$  then 
\begin{equation}
\label{qan}
q(a - n) \ge \frac{1}{C_0^n}q(a) - \sum_{k = 0}^{n-1} \frac{1}{C_0^{k}}.
\end{equation}
This easily follows from condition d) in Assumptions (A).

\section{Lower bound semigroup estimates}

We denote $D= (-2, 2) \times (-2, 2)$. This section concerns estimates of $T_t \mathds{1}_{D}(x)$, which allow to obtain lower bound estimates of $\varphi_1$.

\begin{lem}\label{P_z_dolu}
	Let $a \in \mathbb{R}$, $B_a = \left(a-\frac{1}{2}, a + \frac{1}{2}\right) \times (-\frac{1}{2}, \frac{1}{2})$, $x \in \left(a-\frac{1}{4}, a + \frac{1}{4}\right) \times \left(-\frac{1}{4}, \frac{1}{4}\right)$ and  $0 < t_1 < t_2 < t_0$. Then
	\begin{equation}
	\label{P_z_dolu_inequality}
	P^x \left(X(\tau_{B_a}) \in \frac{D}{2}, t_1 < \tau_{B_a} < t_2\right) \geq \frac{c(t_0) \, (t_2 - t_1)}{(|x_1| + 1)^{1+\alpha}}.
	\end{equation}
\end{lem}

\begin{figure}[ht]
	
	\centering
	
	\includegraphics[scale=0.15]{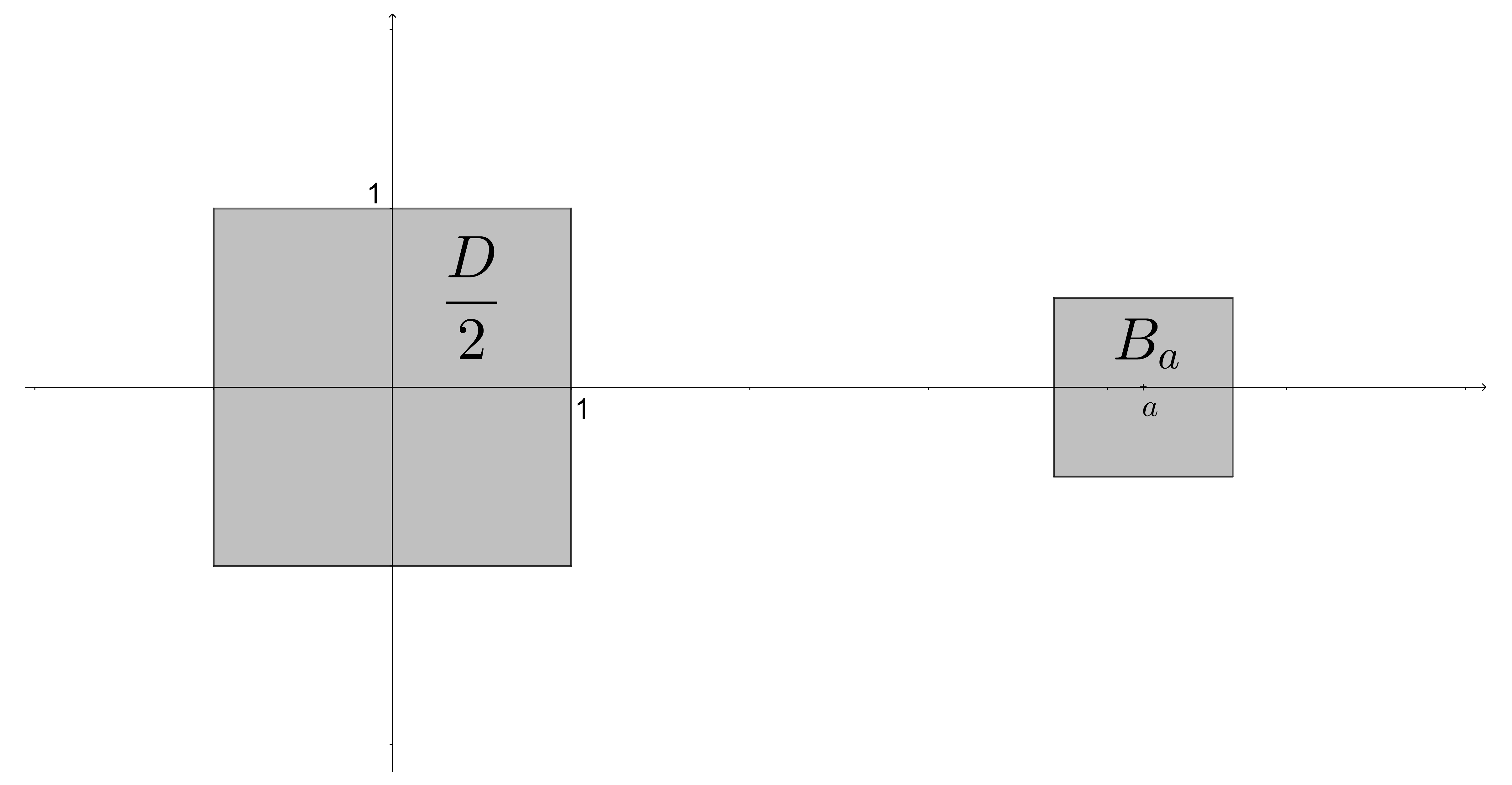}
	\caption{Sets $\frac{D}{2}$ and $B_a$.}\label{rys1}
\end{figure}
\label{rysunek1} 

\begin{proof} We may assume that $a \ge 0$. 
	From the Ikeda-Watanabe formula (\ref{IW_formula}), the left-hand side of (\ref{P_z_dolu_inequality}) is bounded from below by
	\begin{equation}\label{lem1eq1}
	\int_{a-\frac{1}{2}}^{a + \frac{1}{2}} \int_{-\frac{1}{2}}^{\frac{1}{2}} \int_{t_1}^{t_2} p_{B_a} (s, x, y) \, ds \int_{-1}^{-1/2} \frac{\calA}{|y_1 - z_1|^{1+\alpha}} \, dz_1 dy_2 dy_1.
	\end{equation}
	Note that for $y_1 \in \left(a-\frac{1}{2}, a + \frac{1}{2}\right)$ and $z_1 \in (-1, 1)$ we have $|y_1 - z_1| \leq 2 (|x_1| + 1)$. Hence \eqref{lem1eq1} is bounded from below by
	\begin{equation}
	\label{P_z_dolu_formula}
	\frac{c}{(|x_1| + 1)^{1+\alpha}} \int_{t_1}^{t_2} P^x (\tau_{B_a} > s) \, ds.
	\end{equation}
	Note that for $x \in \left(a-\frac{1}{4}, a + \frac{1}{4}\right) \times \left(-\frac{1}{4}, \frac{1}{4}\right)$ we have $\left(x_1 - \frac{1}{4}, x_1 + \frac{1}{4}\right) \times \left(x_2 - \frac{1}{4}, x_2 + \frac{1}{4}\right) \subseteq B_a$. Hence for such $x$ we have 
	\begin{equation*}
	\begin{split}
	P^x (\tau_{B_a} > s) &\geq P^x (\tau_{\left(x_1 - \frac{1}{4}, x_1 + \frac{1}{4}\right) \times \left(x_2 - \frac{1}{4}, x_2 + \frac{1}{4}\right)} > t_0) \\
	&= P^0 (\tau_{\left(- \frac{1}{4}, \frac{1}{4}\right) \times \left(- \frac{1}{4}, \frac{1}{4}\right)} > t_0) = c(t_0).
	\end{split}
	\end{equation*}
	This and (\ref{P_z_dolu_formula}) implies (\ref{P_z_dolu_inequality}).
\end{proof}

\begin{lem}\label{h_z_dolu}
	Let $a \in \mathbb{R}$, $B_a = \left(a-\frac{1}{2}, a + \frac{1}{2}\right) \times (-\frac{1}{2}, \frac{1}{2})$, $x \in \left(a-\frac{1}{4}, a + \frac{1}{4}\right) \times \left(-\frac{1}{4}, \frac{1}{4}\right)$,  $0 < t$. Let us denote
	\begin{equation*}
	h(a, x, t) = E^x \left( X(\tau_{B_a}) \in \frac{D}{2}, \tau_{B_a} < \frac{t}{2}, e^{- \int_0^{\tau_{B_a}} V(X_s) \, ds} \right).
	\end{equation*}
	Then we have
	\begin{equation*}
	h(a, x, t) \geq \frac{c(t)}{(|x_1| + 1)^{1+\alpha} (q(|x_1|) + 1)}.
	\end{equation*}
\end{lem}

\begin{proof}
	Note that
	\begin{equation}\label{h_z_dolu_pomocnicze}
	\begin{split}
	h(a, x, t) &= \sum_{i=1}^{\infty} E^x \left( X(\tau_{B_a}) \in \frac{D}{2}, \frac{t}{2(i+1)} \leq \tau_{B_a} < \frac{t}{2i}, e^{- \int_0^{\tau_{B_a}} V(X_s) \, ds} \right).
	\end{split}
	\end{equation}
	For $\tau_{B_a} < \frac{t}{2i}$ and $X_0 = x$ we have
	\begin{equation*}
	\int_0^{\tau_{B_a}} V(X_s) \, ds \leq \int_0^{\tau_{B_a}} q(|x_1|+1) \, ds \leq \frac{t}{2i} (C_0 q(|x_1|) + C_0).
	\end{equation*}
	From the above inequality and (\ref{h_z_dolu_pomocnicze}) we get
	\begin{equation*}
	h(a, x, t) \geq \sum_{i=1}^{\infty} e^{-\frac{t}{2i} (C_0 q(|x_1|) + C_0)} P^x \left( X(\tau_{B_a}) \in \frac{D}{2}, \frac{t}{2(i+1)} \leq \tau_{B_a} < \frac{t}{2i} \right).
	\end{equation*}
	Using Lemma \ref{P_z_dolu} and then Lemma \ref{series} we obtain
	\begin{equation*}
	\begin{split}
	h(a, x, t) &\geq \sum_{i=1}^{\infty} e^{-\frac{t}{2i} (C_0 q(|x_1|) + C_0)} \, \frac{c_1(t) \, \left( \frac{t}{2i} - \frac{t}{2(i+1)} \right)}{(|x_1| + 1)^{1+\alpha}} \\
	 &\geq \frac{e^{-1} t c_1(t)}{2 (|x_1| + 1)^{1+\alpha} \left( t (C_0 q(|x_1|) + C_0)/2 + 1 \right)}.
	\end{split}
	\end{equation*}
\end{proof}

\begin{lem}\label{k_z_dolu}
	Let $x = (x_1, x_2) \in \mathbb{R}^2$, $|x_2| \geq 1/2$, $0 < t$, 
\begin{eqnarray*} 
U_x &=& \left(x_1- \frac{1}{2}, x_1 + \frac{1}{2}\right) \times \left(-\frac{1}{2}, \frac{1}{2}\right), \,
U_x' = \left(x_1-\frac{1}{4}, x_1 + \frac{1}{4}\right) \times \left(-\frac{1}{4}, \frac{1}{4}\right),\\
W_x &=& \left(x_1- \frac{1}{4}, x_1 + \frac{1}{4}\right) \times \left(x_2- \frac{1}{4}, x_2 + \frac{1}{4}\right).\\ 
\end{eqnarray*}
Let us denote
	\begin{equation*}
	k(x, t) = \begin{multlined}[t] E^x \bigg( X(\tau_{W_x}) \in U_x', \tau_{W_x} < \frac{t}{2}, e^{-\int_0^{\tau_{W_x}} V(X_s) \, ds} \\
 \times E^{X(\tau_{W_x})} \left( X(\tau_{U_x}) \in \frac{D}{2}, \tau_{U_x} < \frac{t}{2}, e^{-\int_0^{\tau_{U_x}} V(X_s) \, ds} \right) \bigg). \end{multlined}
	\end{equation*}
	Then we have
	\begin{equation*}
	k(x, t) \geq \frac{c(t)}{(|x_1| + 1)^{1+\alpha} (|x_2| + 1)^{1+\alpha} (q(|x_1|) + 1) (q(|x|) + 1)}.
	\end{equation*}
\end{lem}

\begin{figure}[ht]
	
	\centering
	
	\includegraphics[scale=0.1]{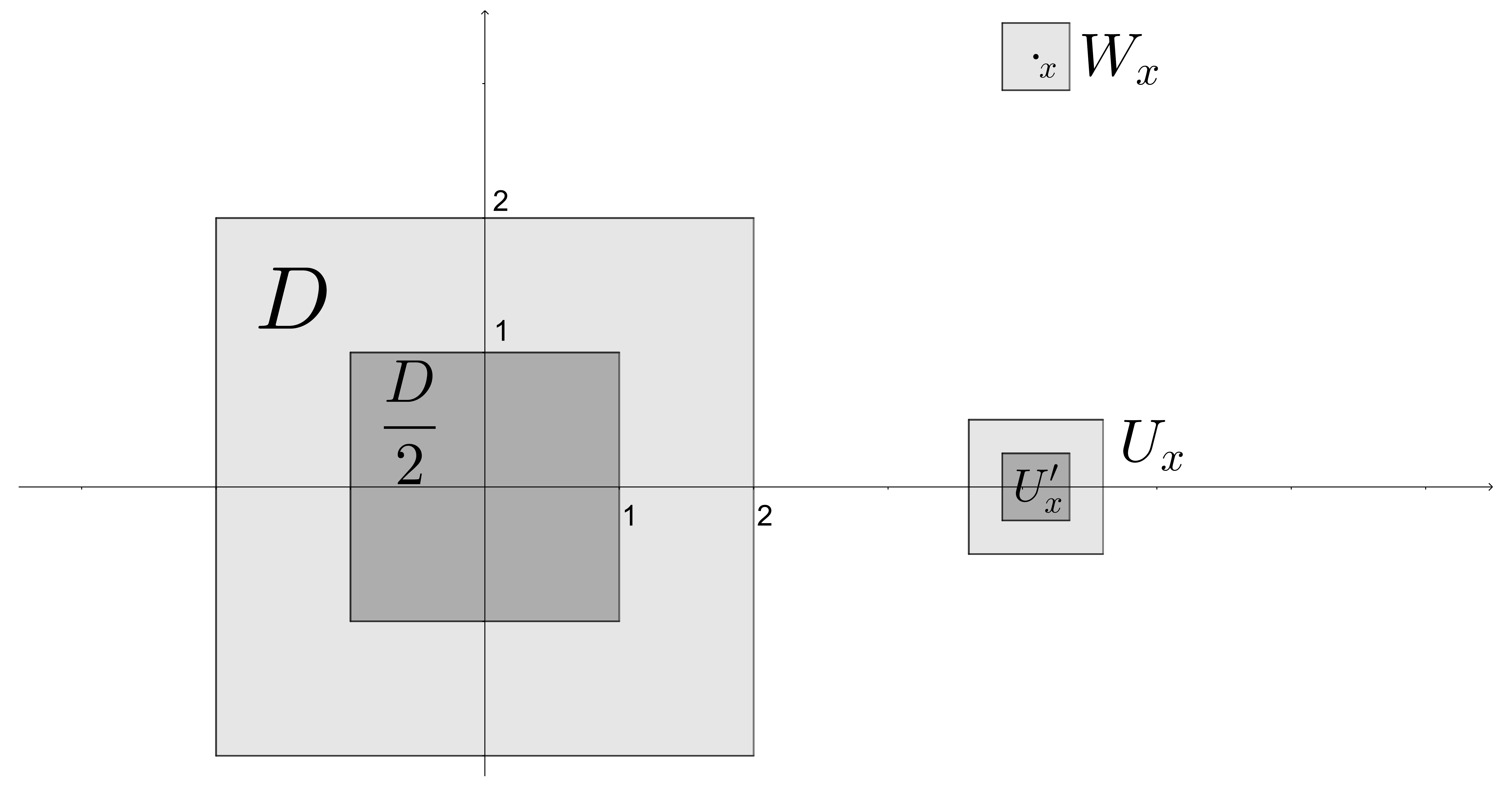}
	\caption{Sets $D$, $W_x$ and $U_x$.}\label{rys2}
\end{figure}
\label{rysunek2}

\begin{proof}
	We denote $X(\tau_{W_x}) = (X(\tau_{W_x})^{(1)}, X(\tau_{W_x})^{(2)})$. From Lemma \ref{h_z_dolu} we get
	\begin{eqnarray}
	\nonumber
	&& k(x, t)	= E^x \Big( X(\tau_{W_x}) \in U_x', \tau_{W_x} < \frac{t}{2}, e^{-\int_0^{\tau_{W_x}} V(X_s) \, ds} \, h (x_1, X(\tau_{W_x}), t) \Big) \\
	\nonumber
	&&\geq E^x \left( X(\tau_{W_x}) \in U_x', \tau_{W_x} < \frac{t}{2}, e^{-\int_0^{\tau_{W_x}} V(X_s) \, ds} \right. \\
	\nonumber
	&& \quad \quad \quad \times \left.
	 \frac{c(t)}{(|X(\tau_{W_x})^{(1)}| + 1)^{1+\alpha} (q(|X(\tau_{W_x})^{(1)}|) + 1)} \right) \\
	\label{expected_value}
	&&\geq \frac{c_1(t)}{(|x_1| + 1)^{1+\alpha} (q(|x_1|) + 1)}   E^x \left( X(\tau_{W_x}) \in U_x', \tau_{W_x} < \frac{t}{2}, e^{-\int_0^{\tau_{W_x}} V(X_s) \, ds} \right). \quad \quad 
	\end{eqnarray}
	The last inequality holds because $|X(\tau_{W_x})^{(1)}| < |x_1| +\frac{1}{4} < |x_1| + 1$.
	
	By similar arguments as in the proof of Lemma \ref{h_z_dolu} we obtain that the expected value in (\ref{expected_value}) is bounded from below by
	$$
	\frac{c_2(t)}{(|x_2| + 1)^{1+\alpha} \left(t (C_0 q(|x|) + C_0) + 1 \right)},
	$$
	which implies the assertion of the lemma.
\end{proof}

\begin{lem}\label{lemat_z_dolu}
Let $x = (x_1, x_2) \in \mathbb{R}^2$, $0 < t$. Then
\begin{equation}
\label{TtD_lower_bound}
T_t \mathds{1}_D(x) \geq \frac{C_3(t)}{(|x_1|+1)^{1+\alpha} (|x_2|+1)^{1+\alpha} (q(\min(|x_1|, |x_2|)) + 1) (q(|x|) + 1)}.
\end{equation}
\end{lem}

\begin{proof}
We will consider 2 cases: 

Case 1. $\max(|x_1|,|x_2|) \ge 1/2$,

Case 2. $\max(|x_1|,|x_2|) < 1/2$.

Let us consider Case 1. We may assume that $|x_1| \le |x_2|$ (the proof for $|x_1| > |x_2|$ is analogous). We will use the notation as in previous lemmas, i.e. $W_x = \left(x_1- \frac{1}{4}, x_1 + \frac{1}{4}\right) \times \left(x_2- \frac{1}{4}, x_2 + \frac{1}{4}\right)$ (this is a square containing the starting point), $U_x = \left(x_1- \frac{1}{2}, x_1 + \frac{1}{2}\right) \times (- \frac{1}{2}, \frac{1}{2})$, $U_x' = \left(x_1-\frac{1}{4}, x_1 + \frac{1}{4}\right) \times \left(-\frac{1}{4}, \frac{1}{4}\right)$. Then
	\begin{eqnarray*}
	&& T_t \mathds{1}_D(x) = E^x\left(X_t \in D, e^{-\int_0^t V(X_s) \, ds}\right) \\
	&& \geq 
	E^x \Big( X(\tau_{W_x}) \in U_x', \tau_{W_x} < \frac{t}{2}, X(\tau_{W_x} + \tau_{U_x} \circ \theta_{\tau_{W_x}}) \in \frac{D}{2}, \tau_{U_x} \circ \theta_{\tau_{W_x}} < \frac{t}{2},\\
	&& \forall s \in [\tau_{W_x} + \tau_{U_x} \circ \theta_{\tau_{W_x}}, t+\tau_{W_x} + \tau_{U_x} \circ \theta_{\tau_{W_x}}] \, X_s \in D,\\
	&& \exp \left( -\int_0^{\tau_{W_x} + \tau_{U_x} \circ \theta_{\tau_{W_x}}} V(X_s) \, ds \right)   \exp \left( -\int_{\tau_{W_x} + \tau_{U_x} \circ \theta_{\tau_{W_x}}}^{t + \tau_{W_x} + \tau_{U_x} \circ \theta_{\tau_{W_x}}} V(X_s) \, ds \right) \bigg).
	\end{eqnarray*}
Now we estimate the last integral in the expression above. We clearly have
	\begin{equation*}
	\int_{\tau_{W_x} + \tau_{U_x} \circ \theta_{\tau_{W_x}}}^{t + \tau_{W_x} + \tau_{U_x} \circ \theta_{\tau_{W_x}}} V(X_s) \, ds \leq \int_0^{2t} V((2,2)) \, ds \leq c(t).
	\end{equation*}
because we have $V(X_s) \in D$ for $s \in [\tau_{W_x} + \tau_{U_x} \circ \theta_{\tau_{W_x}}, t + \tau_{W_x} + \tau_{U_x} \circ \theta_{\tau_{W_x}}]$. Using this and the strong Markov property we get
	\begin{eqnarray*}
	T_t \mathds{1}_D(x) &\geq&  e^{-c(t)} E^x \left( X(\tau_{W_x}) \in U_x', \tau_{W_x} < \frac{t}{2}, e^{-\int_0^{\tau_{W_x}} V(X_s) \, ds} 
	\right. \\
	&& \left. \times E^{X(\tau_{W_x})} \left( X(\tau_{U_x}) \in \frac{D}{2}, \tau_{U_x} < \frac{t}{2}, e^{-\int_0^{\tau_{U_x}} V(X_s) \, ds}
	 P^{X(\tau_{U_x})} (\tau_D > t) \right) \right). 
	\end{eqnarray*}
	
Since for $X(\tau_{U_x}) \in D/2$ we have $(X(\tau_{U_x})^{(1)}-1, X(\tau_{U_x})^{(1)}+1) \times (X(\tau_{U_x})^{(2)}-1, X(\tau_{U_x})^{(2)}+1) \subseteq D$, we get
\begin{eqnarray*}
P^{X(\tau_{U_x})} (\tau_D > t) &\geq& 
P^{X(\tau_{U_x})} (\tau_{(X(\tau_{U_x})^{(1)}-1, X(\tau_{U_x})^{(1)}+1) \times (X(\tau_{U_x})^{(2)}-1, X(\tau_{U_x})^{(2)}+1)} > t) \\
&=& P^0(\tau_{(-1, 1) \times (-1, 1)} > t)\\
&=& c_1(t).
\end{eqnarray*}
Hence
$T_t \mathds{1}_D(x) \geq c_2(t) k(x, t)$,
 where function $k$ is defined in Lemma \ref{k_z_dolu}. From this Lemma we get
\begin{equation}\label{z_dolu_poziom}
T_t \mathds{1}_D(x) \geq \frac{c_3(t)}{(|x_1| + 1)^{1+\alpha} (|x_2| + 1)^{1+\alpha} (q(|x_1|) + 1) (q(|x|) + 1)},
\end{equation}
which implies (\ref{TtD_lower_bound}) in Case 1, when $|x_1| \le |x_2|$.
	
Now, let us consider Case 2. 	Note that the function $x \to T_t \mathds{1}_D(x)$ is continuous and positive, so it attains a positive minimum value on a compact set $[-1/2,1/2] \times [-1/2,1/2]$. It follows that  for any $x \in [-1/2,1/2] \times [-1/2,1/2]$ we have
\begin{equation*}
T_t \mathds{1}_D (x) \geq c_4(t), 
\end{equation*}
which implies (\ref{TtD_lower_bound}) in Case 2.	
\end{proof}

\section{Upper bound semigroup estimates}
In this section we estimate $T_t \mathds{1}_{\mathbb{R}^2}$ from above. This is the most difficult and technical part of the paper. In the whole section we assume that $V$ satisfies \eqref{warunek_logarytmu}. 

We introduce some notation, starting with $1 \times 1$ and $3 \times 3$ squares.
\begin{defn}
	Let $K(x) = ( \lceil x_1 \rceil -2, \lceil x_1 \rceil + 1) \times (\lceil x_2 \rceil -2, \lceil x_2 \rceil + 1)$ and $K_1(x) =( \lceil x_1 \rceil - 1, \lceil x_1 \rceil ] \times (\lceil x_2 \rceil - 1, \lceil x_2 \rceil ]$, where $x = (x_1, x_2) \in \mathbb{R}^2$.
\end{defn}

\begin{defn}
	Let $\sigma_0 = 0$. For $l \in \mathbb{N}$ we define inductively $\sigma_l = \inf \{t \geq \sigma_{l-1}: X_t \notin K(X(\sigma_{l-1}))\}$.
	We define $J(\sigma_l) = X(\sigma_l) - X({\sigma_l}_-)$, where $X({\sigma_l}_-) = \lim_{t \to {\sigma_l}_-} X_t$. We refer to jumps $J(\sigma_l)$ as to $\sigma$ jumps.
\end{defn}
It is clear that each $\sigma$ jump is either vertical or horizontal.

\begin{defn}
	Let $n= (n_1, n_2) \in \mathbb{Z}^2$. We denote $A_n = (n_1 - 2, n_1 + 1) \times (n_2 -2, n_2 + 1) = K(n)$ and $R_n = (n_1 -1, n_1] \times (n_2 - 1, n_2] = K_1(n)$.
\end{defn}

\begin{defn}
	We denote $\tau_n = \tau_{A_n}$.
\end{defn}

By $N$ we denote a constant $N \in \mathbb{N}$ satisfying $N \geq 3$, which may depend on $\alpha$, $V$ and $t$. Its value will be chosen later so that certain inequalities hold. 

\begin{defn}
	Let $\widetilde{R}_N = (-N, N] \times (-N, N]$.
\end{defn}

\begin{defn}\label{def_S}
	Let $k = (k_1, k_2) \in \mathbb{Z}^2$, $l \in \mathbb{N}$ and $t > 0$. Denote $\mathcal{P} = \{p \in \mathbb{Z}^2 : A_p \cap \widetilde{R}_{N} = \emptyset \}$. For $l \ge 2$ let
	\begin{eqnarray*}
	&&S(k, l, t) =\\ 
	 && \bigcup_{p^{(1)} \in \mathcal{P}} \cdots \bigcup_{p^{(l-1)} \in \mathcal{P}} \{X(\sigma_1) \in R_{p^{(1)}}, \dots, X(\sigma_{l-1}) \in R_{p^{(l-1)}}, X(\sigma_l) \in R_k, \sigma_l \le t\}.
	\end{eqnarray*}
	For $l=1$
	\begin{equation*}
	S(k, 1, t) = \{X(\sigma_1) \in R_k, \sigma_1 \le t\}.
	\end{equation*}
\end{defn}

Roughly speaking, $S(k,l,t)$ is an event that the process makes at least $l$ $\sigma$ jumps before time $t$, at time $\sigma_l$ it jumps to square $R_k$ and it does not visit $\widetilde{R}_N$ before $\sigma_l$.
\begin{defn}
	Let $l \in \mathbb{N}$ and $t > 0$. For $l \ge 2$ let
	\begin{eqnarray*}
	&&\widetilde{S} (N + 2, l, t) =\\
	&& \bigcup_{p^{(1)} \in \mathcal{P}} \cdots \bigcup_{p^{(l-1)} \in \mathcal{P}} \{X(\sigma_1) \in R_{p^{(1)}}, \dots, X(\sigma_{l-1}) \in R_{p^{(l-1)}}, X(\sigma_l) \in \widetilde{R}_{N+2}, \sigma_l \le t\}.
	\end{eqnarray*}
	For $l=1$
	\begin{equation*}
	\widetilde{S} (N+2, 1, t) = \{X(\sigma_1) \in \widetilde{R}_{N+2}, \sigma_1 \le t \}.
	\end{equation*}
\end{defn}

Roughly speaking, $\widetilde{S} (N + 2, l, t)$ is an event that the process makes at least $l$ $\sigma$ jumps before time $t$, at time $\sigma_l$ it jumps to square $\widetilde{R}_{N+2}$ and it does not visit $\widetilde{R}_N$ before $\sigma_l$.

\begin{defn}
	Let $n, k \in \mathbb{Z}^2$, $l \in \mathbb{N}$, $m \in \{0, 1, \ldots, l\}$ and $t > 0$. We introduce a function $f(k) = \min_{x \in K(k)} |x|$. If $f(n) > f(k)$ and $m \ge 1$, let us define
	\begin{eqnarray*}
	W(n, m, k, l, t) = &&\bigcup_{p^{(1)} \in P_{1,k}} \cdots \bigcup_{p^{(m-1)} \in P_{1,k}} \bigcup_{p^{(m)} =k} \bigcup_{p^{(m+1)} \in P_{2,k}} \\
	&\cdots& \bigcup_{p^{(l)} \in P_{2,k}} \{X(\sigma_1) \in K_1(p^{(1)}), \dots X(\sigma_l) \in K_1(p^{(l)}), \sigma_l \le t\}, 
	\end{eqnarray*}
	where $P_{1,k} = \{p \in \mathcal{P} : f(p) > f(k)\}$ and $P_{2,k} = \{p \in \mathcal{P} : f(p) \geq f(k) \}$. For $f(n) \leq f(k)$ let $W(n, m, k, l, t) = \emptyset$.
	
	Let us also define
	\begin{equation*}
	W(n, 0, k, l, t) = \bigcup_{p^{(1)} \in P_{2,k}} \cdots \bigcup_{p^{(l)} \in P_{2,k}} \{X(\sigma_1) \in K_1(p^{(1)}), \dots, X(\sigma_l) \in K_1(p^{(l)}), \sigma_l \le t\}
	\end{equation*}
	for $k=n$. If $k \neq n$, let $W(n, 0, k, l, t) = \emptyset$.
\end{defn}

Roughly speaking, for $f(n) > f(k)$  $W(n, m, k, l, t)$ is an event that the process makes at least $l$ $\sigma$ jumps before time $t$, at time $\sigma_m$ the process visits the square $R_k$ and this square is the closest to zero among all squares $R_i$ visited during times $\sigma_1, \ldots, \sigma_l$ and all squares $R_i$ visited during times $\sigma_1, \ldots, \sigma_{m-1}$ are further from zero than the square $R_k$.

\begin{lem}\label{lemat_pstwo_z_gory}
	Let $n = (n_1, n_2) \in \mathbb{Z}^2$, $x = (x_1, x_2) \in R_n$, $k = (k_1, k_2) \in \mathbb{Z}^2$ and $0 < t_1 < t_2$. 
	
	(i) If $k_1 \in \{n_1 - 1, n_1, n_1 + 1\}$ and $|n_2 - k_2| \geq 2$ then
	\begin{equation*}
	P^x (X(\tau_n) \in R_k, t_1 < \tau_n < t_2) \leq \frac{C_4  (t_2 - t_1)}{|n_2 - k_2|^{1+\alpha}}.
	\end{equation*}
	
	(ii) If $k_2 \in \{n_2 - 1, n_2, n_2+1\}$ and $|n_1 - k_1| \geq 2$ then
	\begin{equation*}
	P^x (X(\tau_n) \in R_k, t_1 < \tau_n < t_2) \leq \frac{C_4  (t_2 - t_1)}{|n_1 - k_1|^{1+\alpha}}.
	\end{equation*}
\end{lem}

\begin{figure}[ht]
	
	\centering
	
	\includegraphics[scale=0.1]{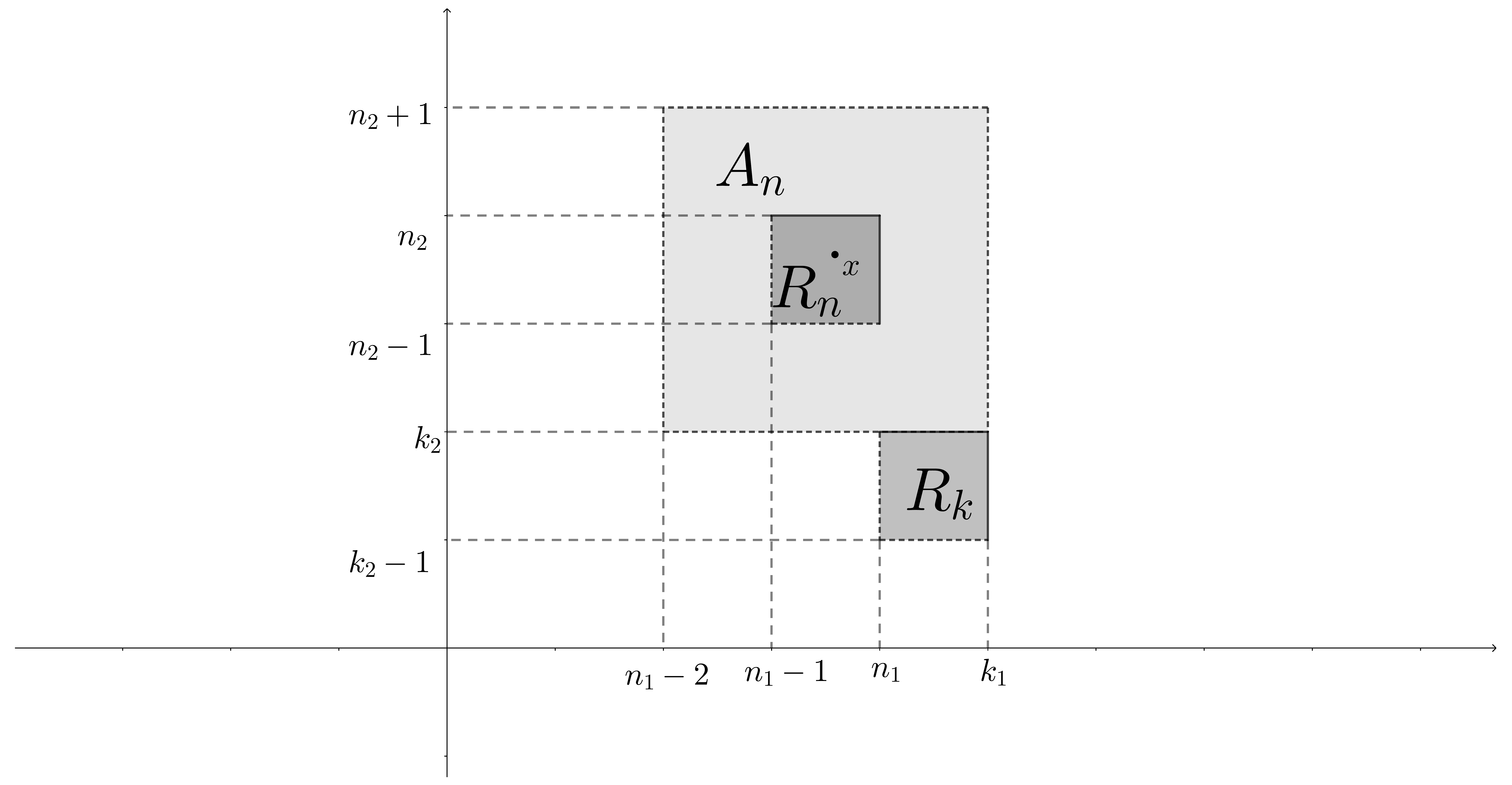}
	\caption{Examples of sets $R_n$, $A_n$ and $R_k$.}\label{rys3}
\end{figure}
\label{rysunek3}

\begin{proof} We will show (i) (the proof of (ii) is similar). 
	First assume $|n_2 - k_2| > 2$. Then, from the Ikeda-Watanabe formula we get
	\begin{eqnarray}
	\nonumber
	&& P^x (X(\tau_n) \in R_k, t_1 < \tau_n < t_2) \\
	\label{one_jump}
	&=& \int_{n_1 - 2}^{n_1 + 1} \int_{n_2 - 2}^{n_2 + 1} \int_{t_1}^{t_2} p_{A_n} (s, x, y) \, ds \int_{k_2 -1}^{k_2} \frac{\calA}{|y_2 - z_2|^{1+\alpha}} \, dz_2 dy_2 dy_1.
	\end{eqnarray}
	For $y_2 \in [n_2 - 2, n_2 + 1]$ and $z_2 \in [k_2 -1, k_2]$ we have
	\begin{equation*}
	\frac{1}{|y_2 - z_2|^{1+\alpha}} \leq \frac{1}{(|n_2 - k_2| - 2)^{1+\alpha}} \leq \frac{3^{1+\alpha}}{|n_2 - k_2|^{1+\alpha}}.
	\end{equation*}
	It follows that (\ref{one_jump}) is bounded from above by
\begin{equation*}
\frac{c_1}{|n_2 - k_2|^{1+\alpha}} \int_{t_1}^{t_2} P^x (\tau_{A_n} > s) \, ds 
\leq \frac{c_1 (t_2 - t_1)}{|n_2 - k_2|^{1+\alpha}}.
\end{equation*}
	
	Now let us consider the case $|n_2 - k_2| =  2$. We may assume that $k_2 = n_2 +2$ (for $k_2 = n_2 -2$ the estimates are similar). Moreover, we may also assume that $k_1 = n_1 -1$ (for $k_1 = n_1$ or $k_1 = n_1 + 1$ the estimates are similar).  By the Ikeda-Watanabe formula we obtain
	\begin{eqnarray}
	\nonumber
	&& P^x (X(\tau_n) \in R_k, t_1 < \tau_n < t_2)  \\
	\label{first_term}
	&=& \int_{n_1 - 2}^{n_1 + 1} \int_{n_2 - 2}^{n_2 + \frac{1}{2}} \int_{t_1}^{t_2} p_{A_n} (s, x, y) \, ds \int_{n_2 +1}^{n_2 + 2} 
	\frac{\calA}{|y_2 - z_2|^{1+\alpha}} \, dz_2 dy_2 dy_1 \\
	\label{second_term}
	&& + \int_{n_1 - 2}^{n_1 + 1} \int_{n_2 + \frac{1}{2}}^{n_2 + 1} \int_{t_1}^{t_2} p_{A_n} (s, x, y) \, ds \int_{n_2 +1}^{n_2 + 2} 
	\frac{\calA}{|y_2 - z_2|^{1+\alpha}} \, dz_2 dy_2 dy_1.
	\end{eqnarray}
	
	First we estimate (\ref{first_term}). For $y_2 \in [n_2 - 2, n_2 + 1/2]$ and $z_2 \in [n_2 +1, n_2+2]$ we have $|y_2 - z_2|^{-1-\alpha} \le 2^{1+\alpha} \le c_2 |n_2 - k_2|^{-1-\alpha}$. Hence (\ref{first_term}) is bounded from above by
	\begin{equation*}
	\frac{c_3}{|n_2 - k_2|^{1+\alpha}} \int_{t_1}^{t_2} P^x (\tau_{A_n} > s) \, ds \leq \frac{c_3 (t_2 - t_1)}{|n_2 - k_2|^{1+\alpha}}.
	\end{equation*}
	
	Now we estimate (\ref{second_term}). For $y_2 \in [n_2 +1/2, n_2 + 1)$ we have
	\begin{equation*}
	\int_{n_2 +1}^{n_2 + 2} \frac{\calA}{|y_2 - z_2|^{1+\alpha}} \, dz_2  
	= \frac{\calA}{\alpha} \left( - (n_2 + 2 - y_2)^{-\alpha} + (n_2 +1 - y_2)^{-\alpha}\right).
	\end{equation*}
	By (\ref{rectangle}) we have $p_{A_n} (s, x, y) = \tilde{p}_{A_{n_1}} (s, x_1, y_1) \tilde{p}_{A_{n_2}} (s, x_2, y_2)$, where $A_{n_1} = (n_1 - 2, n_1 + 1)$ and $A_{n_2} = (n_2 - 2, n_2 + 1)$. For $y_2 \in (n_2 +1/2, n_2 + 1)$ by Theorem 1.1 in \cite{CKS2010} we get
	\begin{equation*}
	\tilde{p}_{A_{n_2}} (s, x_2, y_2) \leq c_4 \frac{\delta_{(n_2 -2, n_2 + 1)} ^ {\frac{\alpha}{2}} (y_2)}{|x_2 - y_2|^{1+\alpha}},
	\end{equation*}
	where $\delta_{(n_2 -2, n_2 + 1)} (y_2) = \text{dist} (y_2, (n_2 -2, n_2 + 1)^c)$.
	Hence (\ref{second_term}) is bounded from above by
	\begin{equation}
	\label{second_term_a}
	 \frac{c_4 \calA}{\alpha} \int_{n_1 - 2}^{n_1 + 1} \int_{t_1}^{t_2} \tilde{p}_{A_{n_1}} (s, x_1, y_1) \int_{n_2 + \frac{1}{2}}^{n_2 + 1}  \frac{\delta_{(n_2 -2, n_2 + 1)} ^ {\frac{\alpha}{2}} (y_2)}{|x_2 - y_2|^{1+\alpha}}   (n_2 +1 - y_2)^{-\alpha} \, dy_2 ds dy_1.
	\end{equation}
	For $y_2 \in (n_2 +1/2, n_2 + 1)$ we have $\delta_{(n_2 -2, n_2 +1)} (y_2) = n_2 +1 - y_2$ and $|x_2 - y_2| \geq \frac{1}{2}$. Therefore, (\ref{second_term_a}) is bounded from above by
	\begin{eqnarray*}
	&& c_5 \int_{n_1 - 2}^{n_1 + 1} \int_{t_1}^{t_2} \tilde{p}_{A_{n_1}} (s, x_1, y_1) \int_{n_2 + \frac{1}{2}}^{n_2 + 1} 
	(n_2 +1 - y_2)^{-\frac{\alpha}{2}}  \, dy_2 ds dy_1 \\
	&=& c_6 \int_{n_1 - 2}^{n_1 + 1} \int_{t_1}^{t_2} \tilde{p}_{A_{n_1}} (s, x_1, y_1) \, ds dy_1 \\
	&\leq& c_6 (t_2 - t_1)\\
	&=& \frac{2^{1+\alpha} c_6   (t_2 - t_1)}{|n_2 - k_2|^{1+\alpha}},
	\end{eqnarray*}
which finishes the proof of the lemma.
\end{proof}

In the sequel we may assume that $N$ is large enough so that for any $n \in Z^2$ such that $A_n \cap \widetilde{R}_N = \emptyset$ we have $q(|n| - 3) \ge 1$.
\begin{lem}\label{lemat_wart_oczek_l_1}
Let $n = (n_1, n_2) \in \mathbb{Z}^2$, $A_n \cap \widetilde{R}_N = \emptyset$, $x = (x_1, x_2) \in R_n$, $k = (k_1, k_2) \in \mathbb{Z}^2$ and $0 < t$.

(i) If $k_1 \in \{n_1 - 1, n_1, n_1 + 1\}$ and $|n_2 - k_2| \geq 2$ then
	\begin{equation*}
	E^x \left(S(k, 1, t), e^{-\frac{1}{2} \int_0^{\sigma_1} V(X_s) \, ds} \right) \leq  \frac{C_5}{|n_2 - k_2|^{1+\alpha} q(|n| - 3)}.
	\end{equation*}
	
(ii) If $k_2 \in \{n_2 - 1, n_2, n_2+1\}$ and $|n_1 - k_1| \geq 2$ then
	\begin{equation*}
	E^x \left(S(k, 1, t), e^{-\frac{1}{2} \int_0^{\sigma_1} V(X_s) \, ds} \right) \leq  \frac{C_5}{|n_1 - k_1|^{1+\alpha} q(|n| - 3)}.
	\end{equation*}	
\end{lem}

\begin{proof} We will show (i) (the proof of (ii) is similar). 
	We have
	\begin{equation*}
	E^x \left(S(k, 1, t), e^{-\frac{1}{2} \int_0^{\sigma_1} V(X_s) \, ds} \right) = E^x \left(X(\tau_n) \in R_k, \tau_n < t, e^{-\frac{1}{2} \int_0^{\sigma_1} V(X_s) \, ds}\right).
	\end{equation*}
	For $X_0 = x$ we have
	\begin{equation*}
	\int_0^{\sigma_1} V(X_s) \, ds \geq \int_0^{\sigma_1} q(|n| -3) \, ds = \tau_n q(|n|-3),
	\end{equation*}
	so
	\begin{equation*}
	\begin{split}
	E^x \Big(X(\tau_n) \in R_k, \tau_n < t, &e^{-\frac{1}{2} \int_0^{\sigma_1} V(X_s) \, ds}\Big) \leq E^x \left(X(\tau_n) \in R_k, \tau_n < t, e^{-\frac{1}{2}\tau_n q(|n|-3)}\right) \\
	&= \sum_{i=1}^{\infty} E^x \left(X(\tau_n) \in R_k, \frac{t}{i+1} \leq \tau_n < \frac{t}{i}, e^{-\frac{1}{2}\tau_n q(|n|-3)}\right) \\
	&\leq \sum_{i=1}^{\infty} e^{-\frac{t}{i+1} \frac{1}{2} q(|n| - 3)}  P^x \left(X(\tau_n) \in R_k, \frac{t}{i+1} \leq \tau_n < \frac{t}{i} \right).
	\end{split}
	\end{equation*}
	From Lemma \ref{lemat_pstwo_z_gory} and then Lemma \ref{series} the expression above is bounded by
	\begin{equation*}
	\sum_{i=1}^{\infty} e^{-\frac{t}{i+1}   \frac{1}{2} q(|n| - 3)} \frac{C_4    \left( \frac{t}{i} - \frac{t}{i+1} \right)}{|n_2 - k_2|^{1+\alpha}} \leq \frac{C_5}{|n_2 - k_2|^{1+\alpha} q(|n| - 3)}.
	\end{equation*}
\end{proof}

Let
\begin{equation*}
C_6(N) = \sup_{p \in \mathbb{Z}^2 \setminus \widetilde{R}_N} \frac{1}{q(|p| - 3)}.
\end{equation*}
Later we choose $N$ large enough so that this constant is sufficiently small.

Put
$$
C_7 = \sum_{j \in \mathbb{Z} \setminus \{0\}} \frac{1}{|j|^{1+\alpha}}.
$$

Now we estimate the expression $q(|p| - 3)$ when $p = (p_1,p_2) \in \mathcal{P}$ and $p_1 = n_1 - 1$ for some $n_1 \in \mathbb{Z}$.
From (\ref{qan}), we get
	\begin{eqnarray}
	\nonumber
	q(|p| - 3) &\geq& \frac{C_0^{-3}}{2} q(|p|) + \frac{C_0^{-3}}{2} q(|p|) - c_1 \\
	\nonumber
	&\geq& \frac{C_0^{-3}}{2} q(|n_1| - 1) + \frac{C_0^{-3}}{2} q(|p|) - c_1 \\
	\nonumber
	&\geq& \frac{C_0^{-4}}{2} q(|n_1|) - \frac{C_0^{-3}}{2} + \frac{C_0^{-3}}{2} q(|p|) - c_1\\
	\label{qpn_estimate}
	&\geq& C_8 q(|n_1|) + 1,
	\end{eqnarray}
where $c_1 = 1 + C_0^{-1} + C_0^{-2}$, $C_8 = C_0^{-4}/2$ and we assumed here that $N$ is chosen large enough so that $(C_0^{-3}/2) q(|p|) - C_0^{-3}/2 - c_1 \geq 1$ for all $p \in \mathcal{P}$. 

\begin{lem}\label{lemat_wart_oczek_S}
	Let $n = (n_1, n_2) \in \mathbb{Z}^2$, $A_n \cap \widetilde{R}_N = \emptyset$ and $k = (k_1, k_2) \in \mathbb{Z}^2$. Assume $x = (x_1, x_2) \in R_n$, $0 < t$ and $l \in \mathbb{N}$, $l \geq 2$. Then we have
	\begin{eqnarray}
	\nonumber
	&&E^x \left(S(k, l, t), e^{-\frac{1}{2} \int_0^{\sigma_l} V(X_s) \, ds}\right) \\
	\label{Sklt}
	&\leq& \frac{2^{l} 3^{3l+\alpha} C_1 C_2^{l-2} C_5^l C_6(N)^{l-2}}{(|n_1 - k_1| + 1)^{1+\alpha} (|n_2 - k_2| + 1)^{1+\alpha} 
	q(|n| -3) (C_8 q(\min(|n_1|,|n_2|)) + 1)}. \quad \quad
	\end{eqnarray}
\end{lem}

\begin{figure}[ht]
	
	\centering
	
	\includegraphics[scale=0.65]{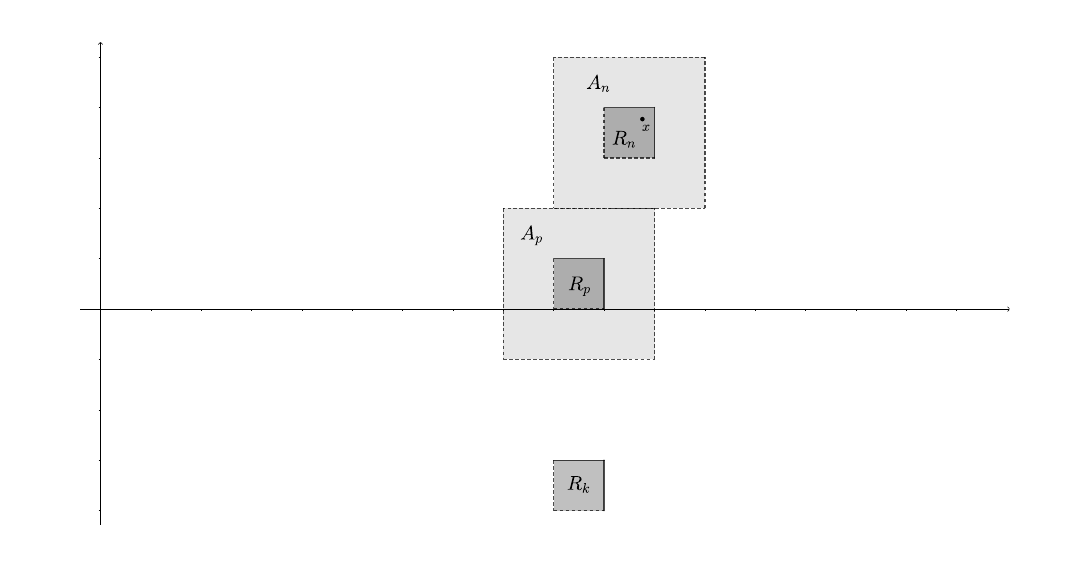}
	\caption{Squares $R_n$, $R_p$ and $R_k$ for $l=2$, both $J(\sigma_1)$, $J(\sigma_2)$ jumps are vertical.}\label{rys4}
\end{figure}
\label{rysunek4}

\begin{proof}
We use induction on $l$. 
	
First, let $l=2$. Note that either both $J(\sigma_1)$, $J(\sigma_2)$ jumps are vertical, or both $J(\sigma_1)$, $J(\sigma_2)$ jumps are horizontal, or one of them is horizontal and the other one is vertical. Hence
\begin{eqnarray*}
&& E^x \left(S(k, 2, t), e^{-\frac{1}{2} \int_0^{\sigma_2} V(X_s) \, ds}\right)\\
&=&  E^x \left(\text{$J(\sigma_1)$, $J(\sigma_2)$ are vertical}, S(k, 2, t), e^{-\frac{1}{2} \int_0^{\sigma_2} V(X_s) \, ds}\right)\\
&& + E^x \left(\text{$J(\sigma_1)$, $J(\sigma_2)$ are horizontal}, S(k, 2, t), e^{-\frac{1}{2} \int_0^{\sigma_2} V(X_s) \, ds}\right)\\
&& + E^x \left(\text{$J(\sigma_1)$ is vertical, $J(\sigma_2)$ is horizontal}, S(k, 2, t), e^{-\frac{1}{2} \int_0^{\sigma_2} V(X_s) \, ds}\right)\\
&& + E^x \left(\text{$J(\sigma_1)$ is horizontal, $J(\sigma_2)$ is vertical}, S(k, 2, t), e^{-\frac{1}{2} \int_0^{\sigma_2} V(X_s) \, ds}\right)\\
&=&  Q_1 + Q_2 + Q_3 + Q_4.
\end{eqnarray*}

Observe that $Q_1$ is equal to
$$
\sum_{p \in \mathcal{P}} E^x \left(\text{$J(\sigma_1)$, $J(\sigma_2)$ are vertical}, X(\sigma_1) \in R_p, X(\sigma_2) \in R_k, \sigma_2 < t, e^{-\frac{1}{2} \int_0^{\sigma_2} V(X_s) \, ds}\right).
$$
	
Recall that $x \in R_n$. Note that the above expected values are not zero only if $|n_1 - k_1| \leq 2$ and  $p$ belongs to one of 3 sets $\mathcal{P}_1 = \{p = (p_1, p_2) \in \mathcal{P} : p_1 = n_1 -1, p_2 \notin \{n_2 - 1, n_2, n_2 + 1, k_2 - 1, k_2, k_2 +1\} \}$, $\mathcal{P}_2 = \{p = (p_1, p_2) \in \mathcal{P} : p_1 = n_1, p_2 \notin \{n_2 - 1, n_2, n_2 + 1, k_2 - 1, k_2, k_2 + 1\} \}$, $\mathcal{P}_3 = \{p = (p_1, p_2) \in \mathcal{P} : p_1 = n_1 + 1, p_2 \notin \{n_2 - 1, n_2, n_2 + 1, k_2 - 1, k_2, k_2 + 1\}\}$. So we have
\begin{equation*}
Q_1 = \sum_{i = 1}^3 \sum_{p \in \mathcal{P}_i} E^x \left(X(\sigma_1) \in R_p, X(\sigma_2) \in R_k, \sigma_2 < t, e^{-\frac{1}{2} \int_0^{\sigma_2} V(X_s) \, ds}\right)
= \sum_{i = 1}^3 Q_{1,i}.
\end{equation*}

We have
\begin{eqnarray*}
Q_{1,1} &=& \sum_{p \in \mathcal{P}_1} E^x \left( S(p, 1, t), X(\sigma_1 + \sigma_1 \circ \theta_{\sigma_1}) \in R_k, \sigma_1 + \sigma_1 \circ \theta_{\sigma_1} < t, \right. \\
 && \left. \quad \quad e^{-\frac{1}{2} \int_0^{\sigma_1} V(X_s) \, ds} e^{-\frac{1}{2} \int_{\sigma_1}^{\sigma_2} V(X_s) \, ds}\right) \\
	&\leq& \sum_{p \in \mathcal{P}_1} E^x \left( S(p, 1, t), X(\sigma_1 + \sigma_1 \circ \theta_{\sigma_1}) \in R_k, \sigma_1 + \sigma_1 \circ \theta_{\sigma_1} < t + \sigma_1, \right. \\
	&& \left. \quad \quad e^{-\frac{1}{2} \int_0^{\sigma_1} V(X_s) \, ds} e^{-\frac{1}{2} \int_{\sigma_1}^{\sigma_2} V(X_s) \, ds} \right)\\
	&=& \sum_{p \in \mathcal{P}_1} E^x \left( S(p, 1, t), e^{-\frac{1}{2} \int_0^{\sigma_1} V(X_s) \, ds} E^{X(\sigma_1)} \left( S(k, 1, t), e^{-\frac{1}{2} \int_0^{\sigma_1} V(X_s) \, ds} \right) \right)
\end{eqnarray*}
By Lemma \ref{lemat_wart_oczek_l_1} this is bounded from above by
	\begin{eqnarray}
	\nonumber
	&& \sum_{p \in \mathcal{P}_1} E^x \left( S(p, 1, t), e^{-\frac{1}{2} \int_0^{\sigma_1} V(X_s) \, ds} \right)   \frac{C_5}{|p_2- k_2|^{1+\alpha} q(|p| - 3)}\\
	\label{sump1n}
	&\le& \sum_{p \in \mathcal{P}_1} \frac{C_5}{|n_2- p_2|^{1+\alpha} q(|n| - 3)}   \frac{C_5}{|p_2- k_2|^{1+\alpha} q(|p| - 3)}.
	\end{eqnarray}
Note that for $p \in \mathcal{P}_1$ we have $p_1 = n_1 - 1$. Using (\ref{qpn_estimate}), Lemma \ref{auxiliary_series} and the inequality $|n_1 - k_1| + 1 \leq 3$ we obtain that (\ref{sump1n}) is bounded from above by
\begin{equation*}
\frac{3^{1+\alpha} C_1 C_5^2}{(|n_1 - k_1| + 1)^{1+\alpha} (|n_2 - k_2| + 1)^{1+\alpha} q(|n|-3) (C_8  q(|n_1|) + 1)}.
\end{equation*}
Similar calculations can be repeated for $Q_{1,2}$ and $Q_{1,3}$ so we get
	\begin{equation}
	\label{pomoc_dwa_pion_z_gory}
	Q_{1} \leq \frac{3^{2+\alpha} C_1 C_5^2}{(|n_1 - k_1| + 1)^{1+\alpha} (|n_2 - k_2| + 1)^{1+\alpha} q(|n|-3) (C_8 q(\min(|n_1|,|n_2|)) + 1)}.
	\end{equation}
	
	The same estimate holds for $Q_2$. Now we estimate $Q_3$ which corresponds to the event that the $J(\sigma_1)$ jump is vertical and 
	the $J(\sigma_2)$ jump is horizontal. 
	
	\begin{figure}[ht]
		\centering
		\includegraphics[scale=0.65]{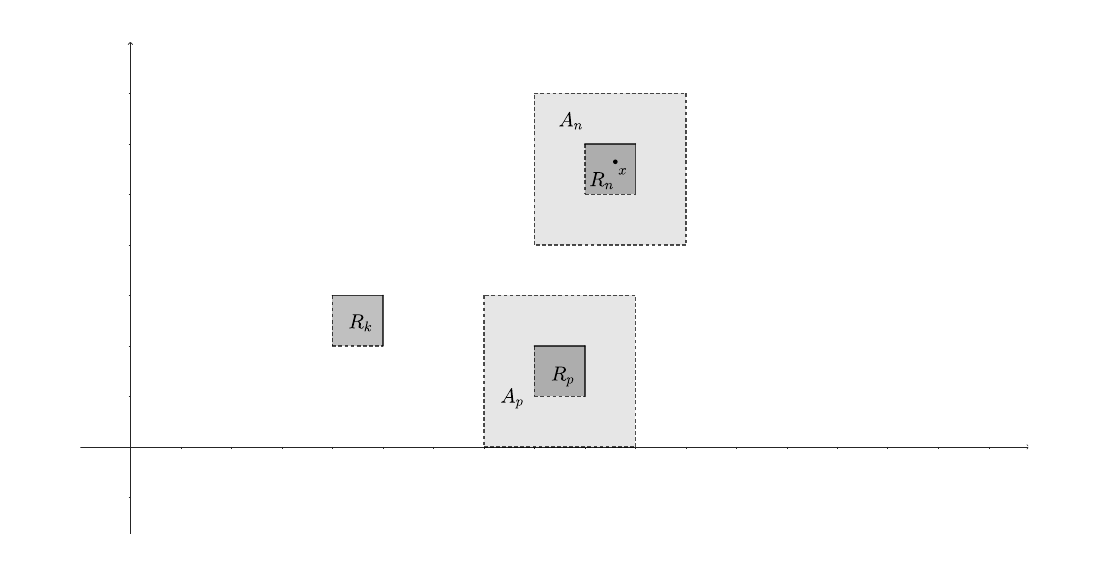}
		\caption{Squares $R_n$, $R_p$ and $R_k$ for $l=2$, $J(\sigma_1)$ jump is vertical, $J(\sigma_2)$ jump is horizontal.}\label{rys5}
	\end{figure}
	\label{rysunek5}
	
Observe that 
\begin{eqnarray*}
Q_3 &=& \sum_{p \in \mathcal{P}} E^x \left(\text{$J(\sigma_1)$ is vertical, $J(\sigma_2)$ is horizontal}, X(\sigma_1) \in R_p, X(\sigma_2) \in R_k, \right. \\
&& \left. \quad \quad \sigma_2 < t, e^{-\frac{1}{2} \int_0^{\sigma_2} V(X_s) \, ds}\right).
\end{eqnarray*}
	
Recall that $x \in R_n$. Note that the above expected values are not zero only if $|n_1 - k_1| \ge 2$, $|n_2 - k_2| \ge 2$ and  $p \in \mathcal{P}_4 = \{p = (p_1, p_2) \in \mathcal{P}: \,  n_1 -1 \leq p_1 \leq n_1 + 1, k_2 - 1 \leq p_2 \leq k_2 + 1 \}$. So we have
\begin{eqnarray*}
Q_3 &=& \sum_{p \in \mathcal{P}_4} E^x \left(X(\sigma_1) \in R_p, X(\sigma_2) \in R_k, \sigma_2 < t, e^{-\frac{1}{2} \int_0^{\sigma_2} V(X_s) \, ds}\right)\\
	&\leq&  \sum_{p \in \mathcal{P}_4} E^x \left( S(p, 1, t), X(\sigma_1 + \sigma_1 \circ \theta_{\sigma_1}) \in R_k, \sigma_1 + \sigma_1 \circ \theta_{\sigma_1} < t + \sigma_1, \right. \\
	&& \quad \quad \quad \left. e^{-\frac{1}{2} \int_0^{\sigma_1} V(X_s) \, ds} e^{- \frac{1}{2}\int_{\sigma_1}^{\sigma_2} V(X_s) \, ds}\right) \\
	&=&\sum_{p \in \mathcal{P}_4} E^x \Big(S(p, 1, t), e^{-\frac{1}{2} \int_0^{\sigma_1} V(X_s) \, ds} E^{X(\sigma_1)} \Big( S(k, 1, t), e^{- \frac{1}{2}\int_0^{\sigma_1} V(X_s) \, ds}\Big) \Big) 
\end{eqnarray*}
	
By Lemma \ref{lemat_wart_oczek_l_1} this is bounded from above by 
	\begin{eqnarray*}
	&& \sum_{p \in \mathcal{P}_4} E^x \left(S(p, 1, t), e^{-\frac{1}{2} \int_0^{\sigma_1} V(X_s) \, ds} \right)   \frac{C_5}{|p_1 - k_1|^{1+\alpha} q(|p| - 3)} \\
	&\leq& \sum_{p \in \mathcal{P}_4} \frac{C_5}{|n_2 - p_2|^{1+\alpha} q(|n| - 3)} \, \frac{C_5}{|p_1 - k_1|^{1+\alpha} q(|p| - 3)}.
	\end{eqnarray*}
	For $p = (p_1,p_2) \in \mathcal{P}_4$ we have $|n_2 - p_2|^{-1-\alpha} \le 3^{1+\alpha} (|n_2 - k_2| + 1)^{-1 - \alpha}$ and similarly $|p_1 - k_1|^{-1-\alpha} \le 3^{1+\alpha} (|n_1 - k_1| + 1)^{-1 - \alpha}$. Moreover, we have $q(|p| - 3) \geq C_8 q(\min(|n_1|,|n_2|)) + 1$, by arguments similar to (\ref{qpn_estimate}).
	Note that the set $\mathcal{P}_4$ has no more than $9$ points. Hence 
\begin{equation}
\label{pomoc_pion_poziom_z_gory}
Q_3 \leq \frac{9 C_5^2 3^{2 + 2 \alpha}}{(|n_1 - k_1| + 1)^{1+\alpha} (|n_2 - k_2| + 1)^{1+\alpha} q(|n|-3) (C_8 q(\min(|n_1|,|n_2|)) + 1)}.
\end{equation}
The same estimate holds for $Q_4$.
	
	From the inequalities \eqref{pomoc_dwa_pion_z_gory} and \eqref{pomoc_pion_poziom_z_gory}  and the same estimates for $Q_2$ and $Q_4$ we get
	\begin{equation*}
	\begin{split}
	E^x &\left(S(k, 2, t), e^{-\frac{1}{2} \int_0^{\sigma_2} V(X_s) \, ds}\right) \\
	&\leq \frac{2^2 3^{6+\alpha} C_1 C_5^2}{(|n_1 - k_1| + 1)^{1+\alpha} (|n_2 - k_2| + 1)^{1+\alpha} q(|n|-3) (C_8 q(\min(|n_1|,|n_2|)) + 1)}.
	\end{split}
	\end{equation*}
	So, the assertion of the Lemma holds for $l=2$. Recall that $C_1 \ge 1$, we can also assume that $C_5 \ge 1$.
	
	Now assume that (\ref{Sklt}) holds for $l-1$ (where $l \geq 3$), that is
	\begin{equation*}
	\begin{multlined}[t]
	E^x \left(S(k, l-1, t), e^{-\frac{1}{2} \int_0^{\sigma_{l-1}} V(X_s) \, ds}\right) \\
	\leq \frac{2^{l-1} 3^{3(l-1)+\alpha} C_1 C_2^{l-3} C_5^{l-1} C_6(N)^{l-3}}{(|n_1 - k_1| + 1)^{1+\alpha} (|n_2 - k_2| + 1)^{1+\alpha} 
	q(|n| -3) (C_8 q(\min(|n_1|,|n_2|)) + 1)}.
	\end{multlined}
	\end{equation*}
	We show that (\ref{Sklt}) holds for $l$ (where $l \geq 3$).
	
	Observe that 
\begin{eqnarray}
\nonumber
&& E^x \left(S(k, l, t), e^{-\frac{1}{2} \int_0^{\sigma_{l}} V(X_s) \, ds}\right)\\
\label{sum_p_5_10}
&=& \sum_{p \in \mathcal{P}} E^x \left(S(k, l, t), X(\sigma_{l-1}) \in R_p, e^{-\frac{1}{2} \int_0^{\sigma_{l}} V(X_s) \, ds}\right).
\end{eqnarray}
	
Note that the above expected values are not zero only if $p$ belongs to one of 6 sets $\mathcal{P}_5 = \{p = (p_1, p_2) \in \mathcal{P} : p_1 = k_1 -1, |p_2 - k_2| > 1 \}$, $\mathcal{P}_6 = \{p = (p_1, p_2) \in \mathcal{P} : p_1 = k_1, |p_2 - k_2| > 1 \}$, $\mathcal{P}_7 = \{p = (p_1, p_2) \in \mathcal{P} : p_1 = k_1 +1, |p_2 - k_2| > 1 \}$, 
$\mathcal{P}_8 = \{p = (p_1, p_2) \in \mathcal{P} : p_2 = k_2 -1, |p_1 - k_1| > 1 \}$, $\mathcal{P}_9 = \{p = (p_1, p_2) \in \mathcal{P} : p_2 = k_2, |p_1 - k_1| > 1 \}$, $\mathcal{P}_{10} = \{p = (p_1, p_2) \in \mathcal{P} : p_2 = k_2 +1, |p_1 - k_1| > 1 \}$.

Hence (\ref{sum_p_5_10}) is equal to
\begin{eqnarray}
\sum_{i = 5}^{10} \sum_{p \in \mathcal{P}_i} E^x \left(S(k, l, t), X(\sigma_{l-1}) \in R_p, e^{-\frac{1}{2} \int_0^{\sigma_{l}} V(X_s) \, ds}\right)
= \sum_{i = 5}^{10} Q_{i}.
\end{eqnarray}
	
$Q_5$ is equal to
	\begin{eqnarray*}
	 &&\sum_{p \in \mathcal{P}_5}  E^x \left(S(p, l-1, t), X(\sigma_{l-1} + \sigma_1 \circ \theta_{\sigma_{l-1}}) \in R_k, \sigma_{l-1} + \sigma_1 \circ \theta_{\sigma_{l-1}} < t, \right. \\
	&& \quad \quad \quad \left.
	e^{-\frac{1}{2} \int_0^{\sigma_{l-1}} V(X_s) \, ds} e^{- \frac{1}{2}\int_{\sigma_{l-1}}^{\sigma_l} V(X_s) \, ds}\right) \\
	&\le& \sum_{p \in \mathcal{P}_5} E^x \left(S(p, l-1, t), e^{-\frac{1}{2} \int_0^{\sigma_{l-1}} V(X_s) \, ds} E^{X(\sigma_{l-1})} \left(S(k, 1, t), e^{-\frac{1}{2} \int_0^{\sigma_1} V(X_s) \, ds} \right) \right).
	\end{eqnarray*}
Using Lemma \ref{lemat_wart_oczek_l_1} and the induction hypothesis this is bounded from above by
	\begin{eqnarray}
	\nonumber
	  && \sum_{p \in \mathcal{P}_5} \frac{2^{l-1} 3^{3(l-1)+\alpha} C_1 C_2^{l-3} C_5^l C_6(N)^{l-3}}{(|n_1 - p_1| + 1)^{1+\alpha} 
	q(|n| -3) (C_8 q(\min(|n_1|,|n_2|)) + 1) q(|p| - 3)}\\
	\label{Q5}
	&\times&
	\frac{1}{(|n_2 - p_2| + 1)^{1+\alpha} |p_2 - k_2|^{1+\alpha}} 
	\end{eqnarray}
	
	Note that for $p \in \mathcal{P}_5$ we have $p_1 = k_1 - 1$, so $(|n_1 - p_1| + 1)^{-1-\alpha} = (|n_1 - k_1 + 1| + 1)^{-1-\alpha} \leq 3^2 (|n_1 - k_1| + 1)^{-1-\alpha}$. We also have $\frac{1}{q(|p|-3)} \leq C_6(N)$. Using this and Lemma \ref{auxiliary_series} we obtain that (\ref{Q5}) is bounded from above by
$$
 \frac{2^{l-1} 3^{3(l-1)+\alpha+2} C_1 C_2^{l-2} C_5^l C_6(N)^{l-2}}{(|n_1 - k_1| + 1)^{1+\alpha} (|n_2 - k_2| + 1)^{1+\alpha} 
	q(|n| -3) (C_8 q(\min(|n_1|,|n_2|)) + 1)}.
$$
	
This gives the estimate of $Q_5$. Estimates of $Q_6$, $Q_7$, $Q_8$, $Q_9$ and $Q_{10}$ are similar. Hence, we obtain
	\begin{equation*}
	\begin{multlined}[t]
	E^x \left(S(k, l, t), e^{-\frac{1}{2} \int_0^{\sigma_l} V(X_s) \, ds}\right) \\
	\leq \frac{2^{l} 3^{3l+\alpha} C_1 C_2^{l-2} C_5^l C_6(N)^{l-2}}{(|n_1 - k_1| + 1)^{1+\alpha} (|n_2 - k_2| + 1)^{1+\alpha} q(|n| -3) 
	(C_8 q(\min(|n_1|,|n_2|)) + 1)},
	\end{multlined}
	\end{equation*}
which finishes the proof by induction.
\end{proof}

\begin{lem}\label{wniosek_S_z_fala}
	Let $n = (n_1, n_2) \in \mathbb{Z}^2$, $A_n \cap \widetilde{R}_N = \emptyset$, $t > 0$ and $x = (x_1, x_2) \in R_n$.  For $l \in \mathbb{N}$, $l \ge 2$ we have
	\begin{equation*}
	\begin{multlined}[t]
	E^x \left(\widetilde{S}(N+2, l, t), e^{-\frac{1}{2} \int_0^{\sigma_l} V(X_s) \, ds}\right) \\
	\leq \frac{2^{l} 3^{3l+\alpha} C_1 C_2^{l-2} C_5^{l} C_6(N)^{l-2} \widetilde{C}_5 (N)}{(||n_1| - N| + 1)^{1+\alpha} (||n_2| - N| + 1)^{1+\alpha} q(|n| -3) (C_8 q(\min(|n_1|,|n_2|)) + 1)}.
	\end{multlined}
	\end{equation*}
	We also have
	\begin{equation*}
	E^x \left( \widetilde{S}(N+2, 1, t), e^{-\frac{1}{2} \int_0^{\sigma_1} V(X_s) \, ds}\right) \leq \frac{\widetilde{C}_5 (N)}{(|n_2|-N)^{1+\alpha} q(|n| - 3)}
	\end{equation*}
	for $-N-2 \leq n_1 \leq N+3$ and ($n_2 > N+5$ or $n_2 < -N-4$). Similarly
	\begin{equation*}
	E^x \left( \widetilde{S}(N+2, 1, t), e^{-\frac{1}{2} \int_0^{\sigma_1} V(X_s) \, ds}\right) \leq \frac{\widetilde{C}_5 (N)}{(|n_1|-N)^{1+\alpha} q(|n| - 3)}
	\end{equation*}
	for $-N-2 \leq n_2 \leq N+3$ and ($n_1 > N+5$ or $n_1 < -N-4$).
\end{lem}
We omit the proof of this result. It follows by similar arguments as in proofs of Lemmas \ref{lemat_wart_oczek_l_1} and \ref{lemat_wart_oczek_S}. Note that the volume of $\widetilde{R}_{N+2}$ depends on $N$ so that the constant $\widetilde{C}_5 (N)$ appears on the right-hand side of the inequalities in the assertion of Lemma \ref{wniosek_S_z_fala}.

\begin{lem}\label{lemat_wart_oczek_W}
	For $n = (n_1, n_2) \in \mathbb{Z}^2$, $A_n \cap \widetilde{R}_N = \emptyset$, $k = (k_1, k_2) \in \mathcal{P}$, $x \in R_n$, $m, l \in \mathbb{N}$ such that $2 \leq m \leq l$ and $0< t$ we have
	\begin{equation*}
	\begin{multlined}[t]
	E^x \left( W(n, m, k, l, t), e^{-\frac{1}{2} \int_0^{\sigma_l} V(X_s) \, ds} \right) \\
	\leq \frac{2^{l} 3^{l+2m+\alpha} C_1 C_2^{m-2} C_5^l C_6(N)^{l-2}  C_7^{l-m}}{(|n_1 - k_1| + 1)^{1+\alpha} (|n_2 - k_2| + 1)^{1+\alpha} q(|n|-3) (C_8 q(\min(|n_1|,|n_2|)) + 1)}.
	\end{multlined}
	\end{equation*}
\end{lem}

\begin{proof}
	The proof is by induction on $l$.
	
	First note that
	\begin{equation*}
	E^x \Big( W(n, l, k, l, t), e^{-\frac{1}{2} \int_0^{\sigma_l} V(X_s) \, ds} \Big) \leq E^x \Big( S(k, l, t), e^{-\frac{1}{2} \int_0^{\sigma_l} V(X_s) \, ds} \Big)
	\end{equation*}
	so from Lemma \ref{lemat_wart_oczek_S} it follows that
	\begin{equation*}
	\begin{multlined}[t]
	E^x \Big( W(n, l, k, l, t), e^{-\frac{1}{2} \int_0^{\sigma_l} V(X_s) \, ds} \Big) \\
	\leq \frac{2^{l} 3^{3l+\alpha} C_1 C_2^{l-2} C_5^l C_6(N)^{l-2}}{(|n_1 - k_1| + 1)^{1+\alpha} (|n_2 - k_2| + 1)^{1+\alpha} 
	q(|n| -3) (C_8 q(\min(|n_1|,|n_2|)) + 1)}.
	\end{multlined}
	\end{equation*}
	Therefore, the inequality holds for $l = m$.
	
	Assume that the inequality holds for some $l \ge m$. We show that it holds for $l+1$. Let $P_{2,k} = \{ p \in \mathcal{P} : f(p) \geq f(k) \}$ (as in the definition of $W$). By the strong Markov property we obtain
	\begin{equation*}
	\begin{split}
	& E^x \Big( W(n, m, k, l+1, t), e^{-\frac{1}{2} \int_0^{\sigma_{l+1}} V(X_s) \, ds} \Big) \\
	&= \begin{multlined}[t] \sum_{p \in P_{2,k}} E^x \Big( W(n, m, k, l, t), X(\sigma_l + \sigma_1 \circ \theta_{\sigma_l}) \in K_1 (p), \sigma_l + \sigma_1 \circ \theta_{\sigma_l} <t,\\
	e^{-\frac{1}{2} \int_0^{\sigma_{l}} V(X_s) \, ds} e^{-\frac{1}{2} \int_{\sigma_l}^{\sigma_{l+1}} V(X_s) \, ds} \Big) \end{multlined} \\
	&\leq \begin{multlined}[t] \sum_{p \in P_{2,k}} E^x \Big( W(n, m, k, l, t), X(\sigma_l + \sigma_1 \circ \theta_{\sigma_l}) \in K_1 (p), \sigma_l + \sigma_1 \circ \theta_{\sigma_l} <t + \sigma_l,\\
	e^{-\frac{1}{2} \int_0^{\sigma_{l}} V(X_s) \, ds} e^{-\frac{1}{2} \int_{\sigma_l}^{\sigma_{l+1}} V(X_s) \, ds} \Big) \end{multlined} \\
	&= E^x \Big( W(n, m, k, l, t), e^{-\frac{1}{2} \int_0^{\sigma_{l}} V(X_s) \, ds} \sum_{p \in P_{2,k}} E^{X(\sigma_l)} \left( S(p, 1, t), e^{-\frac{1}{2} \int_0^{\sigma_1} V(X_s) \, ds}\right) \Big).
	\end{split}
	\end{equation*}
	
We denote $X(\sigma_l) = (X(\sigma_l)^{(1)}, X(\sigma_l)^{(2)})$. For $i \in \{-1,0,1\}$ put $\widetilde{\mathcal{P}}_i = \{p = (p_1, p_2) \in P_{2,k}: p_1 = \lceil X(\sigma_l)^{(1)} \rceil - i \quad \text{and} \quad |p_2 - \lceil X(\sigma_l)^{(2)} \rceil | > 1 \}$ and  $\widetilde{\widetilde{\mathcal{P}}}_i = \{p = (p_1, p_2) \in P_{2,k}: p_2 = \lceil X(\sigma_l)^{(2)} \rceil - i \quad \text{and} \quad |p_1 - \lceil X(\sigma_l)^{(1)} \rceil| > 1 \}$. Note that $S(p, 1, t)$ is not an empty set if $p \in \widetilde{\mathcal{P}}_i$ or $p \in \widetilde{\widetilde{\mathcal{P}}}_i$ for some $i \in \{-1,0,1\}$.

Note that $(\lceil X(\sigma_l)^{(1)} \rceil, \lceil X(\sigma_l)^{(2)} \rceil) \in \mathcal{P}$. Hence, by Lemma \ref{lemat_wart_oczek_l_1} we get 
	\begin{equation*}
	\begin{split}
	&\sum_{p \in \widetilde{\mathcal{P}}_{-1}} E^{X(\sigma_l)} \Big( S(p, 1, t), e^{-\frac{1}{2} \int_0^{\sigma_1} V(X_s) \, ds}\Big) \\
	&\leq \sum_{p \in \widetilde{\mathcal{P}}_{-1}} \frac{C_5}{(\lceil X(\sigma_l)^{(2)} \rceil - p_2)^{1+\alpha} 
	q(|(\lceil X(\sigma_l)^{(1)} \rceil, \lceil X(\sigma_l)^{(2)} \rceil )| - 3)} \\
	&\leq C_5 C_6(N) \sum_{p \in \widetilde{\mathcal{P}}_{-1}} \frac{1}{(\lceil X(\sigma_l)^{(2)} \rceil - p_2)^{1+\alpha}} \\
	&\leq C_5 C_6(N) C_7,
	\end{split}
	\end{equation*}
In the last inequality we used the fact that $C_7 = \sum_{j \in \mathbb{Z} \setminus \{0\}} \frac{1}{|j|^{1+\alpha}}$.
	
Using this and analogous estimates for $\widetilde{\mathcal{P}}_{0}$, $\widetilde{\mathcal{P}}_{1}$, 
$\widetilde{\widetilde{\mathcal{P}}}_{-1}$, $\widetilde{\widetilde{\mathcal{P}}}_{0}$, $\widetilde{\widetilde{\mathcal{P}}}_{1}$ we get
	\begin{eqnarray*}
	&& E^x \Big( W(n, m, k, l, t), e^{-\frac{1}{2} \int_0^{\sigma_{l}} V(X_s) \, ds} \sum_{p \in P_{2,k}} E^{X(\sigma_l)} \left( S(p, 1, t), e^{-\frac{1}{2} \int_0^{\sigma_1} V(X_s) \, ds}\right) \Big) \\
	&\leq& 	E^x \Big( W(n, m, k, l, t), e^{-\frac{1}{2} \int_0^{\sigma_{l}} V(X_s) \, ds}   6 C_5 C_6(N) C_7 \Big).
	\end{eqnarray*}
	From the induction hypothesis this is bounded by
	\begin{equation*}
	\begin{multlined}[t]
	\frac{2^{l} 3^{l+2m+\alpha} C_1 C_2^{m-2} C_5^l C_6(N)^{l-2} C_7^{l-m}   6 C_5 C_6(N) C_7}{(|n_1 - k_1| + 1)^{1+\alpha} (|n_2 - k_2| + 1)^{1+\alpha} q(|n|-3) (C_8 q(\min(|n_1|,|n_2|)) + 1)} \\
	= \frac{2^{l+1} 3^{l+2m+\alpha+1} C_1 C_2^{m-2} C_5^{l+1} C_6(N)^{l-1} C_7^{l+1-m}}{(|n_1 - k_1| + 1)^{1+\alpha} (|n_2 - k_2| + 1)^{1+\alpha} q(|n|-3) (C_8 q(\min(|n_1|,|n_2|)) + 1)},
	\end{multlined}
	\end{equation*}
which finishes the proof. 
\end{proof}

\begin{lem}
\label{lemat_wart_oczek_W_1}
	For $n = (n_1, n_2) \in \mathbb{Z}^2$, $A_n \cap \widetilde{R}_N = \emptyset$, $k = (k_1, k_2) \in \mathcal{P}$, $x \in R_n$, $l \in \mathbb{N}$ and $l \geq 2$, $0 < t$ we have
	\begin{equation*}
	\begin{multlined}[t]
	E^x \left( W(n, 1, k, l, t), e^{-\frac{1}{2} \int_0^{\sigma_l} V(X_s)\, ds} \right) \\
	\leq \frac{2^{l} 3^{l+2+\alpha} C_5^l C_6(N)^{l-2} C_7^{l-1}}{(|n_1 - k_1| + 1)^{1+\alpha} (|n_2 - k_2| + 1)^{1+\alpha} 
	q(|n|-3) (C_8 q(\min(|n_1|,|n_2|)) + 1)}.
	\end{multlined}
	\end{equation*}
\end{lem}

\begin{proof}
	Note that the left side of the inequality is not zero only if $k_1 \in \{n_1 -1, n_1, n_1 + 1\}$ and $|k_2 - n_2| \ge 2$ or if $k_2 \in \{n_2 -1, n_2, n_2 + 1\}$ and $|k_1 - n_1| \ge 2$. The calculations are similar for each case. Therefore we may assume that $k_2 = n_2$ and 
	$|k_1 - n_1| \ge 2$.
	
First we will prove the lemma for $l = 2$. Using similar arguments as in the proof of Lemma \ref{lemat_wart_oczek_W} we obtain
\begin{eqnarray}
\nonumber
&& E^x \left( W(n, 1, k, 2, t), e^{-\frac{1}{2} \int_0^{\sigma_2} V(X_s)\, ds} \right) \\
\label{Wn1k1ta}
&\leq& E^x \left( W(n, 1, k, 1, t), X(\sigma_1) \in R_k, e^{-\frac{1}{2} \int_0^{\sigma_1} V(X_s)\, ds} \right. \\
\label{Wn1k1tb}
&& \quad \quad \times \left. \sum_{p \in \mathcal{P}} E^{X(\sigma_1)} \left(S(p, 1, t), e^{-\frac{1}{2} \int_0^{\sigma_1} V(X_s)\, ds} \right) \right).
\end{eqnarray}
	
Note that the above expected values are not zero only if $p$ belongs to one of 6 sets $\mathcal{P}_1^k = \{p = (p_1, p_2) \in \mathcal{P} : p_1 = k_1 -1, |p_2 - k_2| > 1 \}$, $\mathcal{P}_2^k = \{p = (p_1, p_2) \in \mathcal{P} : p_1 = k_1, |p_2 - k_2| > 1 \}$, $\mathcal{P}_3^k = \{p = (p_1, p_2) \in \mathcal{P} : p_1 = k_1 +1, |p_2 - k_2| > 1 \}$, $\mathcal{P}_4^k = \{p = (p_1, p_2) \in \mathcal{P} : p_2 = k_2 -1, |p_1 - k_1| > 1 \}$, $\mathcal{P}_5^k = \{p = (p_1, p_2) \in \mathcal{P} : p_2 = k_2, |p_1 - k_1| > 1 \}$, $\mathcal{P}_{6}^k = \{p = (p_1, p_2) \in \mathcal{P} : p_2 = k_2 +1, |p_1 - k_1| > 1 \}$.

Hence (\ref{Wn1k1ta}-\ref{Wn1k1tb}) is equal to
\begin{eqnarray*}
&& E^x \left( W(n, 1, k, 1, t), X(\sigma_1) \in R_k, e^{-\frac{1}{2} \int_0^{\sigma_1} V(X_s)\, ds} \right. \\
&& \quad \quad \times \left. \sum_{i = 1}^{6} \sum_{p \in \mathcal{P}_i^k} E^{X(\sigma_1)} \left(S(p, 1, t), e^{-\frac{1}{2} \int_0^{\sigma_1} V(X_s)\, ds} \right) \right).
\end{eqnarray*}

For $X(\sigma_1) \in R_k$ we obtain by Lemma \ref{lemat_wart_oczek_l_1}, the assumption that $k_2 = n_2$ and arguments similar to (\ref{qpn_estimate})
	\begin{eqnarray*}
	&& \sum_{p \in \mathcal{P}_1^k} E^{X(\sigma_1)} \left( S(p, 1, t), e^{-\frac{1}{2} \int_0^{\sigma_1} V(X_s)\, ds} \right) \\
	&\leq& \sum_{p \in \mathcal{P}_1^k} \frac{C_5}{|k_2 - p_2|^{1+\alpha} q(|k| -3)} \\
	&\leq& \sum_{p \in \mathcal{P}_1^k} \frac{C_5}{|k_2 - p_2|^{1+\alpha} (C_8 q(\min(|n_1|,|n_2|)) + 1)} \\
	&\leq& \frac{C_5 C_7}{C_8 q(\min(|n_1|,|n_2|)) + 1}.
	\end{eqnarray*}
Simliar estimates hold for $\mathcal{P}_2^k$, $\mathcal{P}_3^k$, $\mathcal{P}_4^k$, $\mathcal{P}_5^k$, $\mathcal{P}_6^k$. Using this and again Lemma \ref{lemat_wart_oczek_l_1} we get
	\begin{equation*}
	\begin{split}
	E^x \Big( &W(n, 1, k, 2, t), e^{-\frac{1}{2} \int_0^{\sigma_2} V(X_s)\, ds} \Big) \\
	&\leq E^x \left( W(n, 1, k, 1, t), e^{-\frac{1}{2} \int_0^{\sigma_1} V(X_s)\, ds} \right)  \frac{6C_5 C_7}{C_8 q(\min(|n_1|,|n_2|)) + 1} \\
	&\leq E^x \left( S(k, 1, t), e^{-\frac{1}{2} \int_0^{\sigma_1} V(X_s)\, ds} \right)  \frac{6C_5 C_7}{C_8 q(\min(|n_1|,|n_2|)) + 1} \\
	&\leq \frac{C_5}{q(|n| - 3) |n_1 - k_1|^{1+\alpha}}  \frac{6C_5 C_7}{C_8 q(\min(|n_1|,|n_2|)) + 1} \\
	&\leq \frac{2^1 3^{4+\alpha} C_5^2 C_7}{(|n_1 - k_1| + 1)^{1+\alpha} (|n_2 - k_2| + 1)^{1+\alpha}
	q(|n| - 3) (C_8 q(\min(|n_1|,|n_2|)) + 1)},
	\end{split}
	\end{equation*}
	so the lemma holds for $l = 2$.
	
Now we use induction. Assume that the assertion of the lemma holds for some $l$. By the same arguments as in the proof of Lemma \ref{lemat_wart_oczek_W}, we get
	\begin{equation*}
	\begin{multlined}[t]
	E^x \left( W(n, 1, k, l + 1, t), e^{-\frac{1}{2} \int_0^{\sigma_{l+1}} V(X_s)\, ds} \right) \\
	\leq E^x \Big( W(n, 1, k, l, t), e^{-\frac{1}{2} \int_0^{\sigma_{l}} V(X_s) \, ds} \Big) 6 C_5 C_6(N) C_7.
	\end{multlined}
	\end{equation*}
	From the induction hypothesis this is bounded from above by
	\begin{equation*}
	\frac{2^{l+1} 3^{l+3+\alpha} C_5^{l+1} C_6(N)^{l-1} C_7^l}{(|n_1 - k_1| + 1)^{1+\alpha} (|n_2 - k_2| + 1)^{1+\alpha} 
	q(|n|-3) (C_8 q(\min(|n_1|,|n_2|)) + 1)}
	\end{equation*}
	so the the assertion of the lemma holds for $l+1$.
\end{proof}

By Lemmas \ref{lemat_wart_oczek_W}, \ref{lemat_wart_oczek_W_1} and the fact that $C_2 \ge 1$ we easily obtain the following conlusion.
\begin{cor}\label{wniosek_szacowanie_W}
	For $n = (n_1, n_2) \in \mathbb{Z}^2$, $A_n \cap \widetilde{R}_N = \emptyset$, $k = (k_1, k_2) \in \mathcal{P}$, $x \in R_n$, $m, l \in \mathbb{N}$ such that $1 \leq m$, $2 \leq l$, $l \geq m$ and $0< t$  we have
	\begin{equation*}
	\begin{multlined}[t]
	E^x \left( W(n, m, k, l, t), e^{-\frac{1}{2} \int_0^{\sigma_l} V(X_s) \, ds} \right) \\
	\leq \frac{2^{l} 3^{l+2m+\alpha} C_1 C_2^{m-1} C_5^l C_6(N)^{l-2} C_7^{l-m}}{(|n_1 - k_1| + 1)^{1+\alpha} (|n_2 - k_2| + 1)^{1+\alpha} 
	q(|n|-3) (C_8 q(\min(|n_1|,|n_2|)) + 1)}.
	\end{multlined}
	\end{equation*}
\end{cor}

\begin{lem}\label{lemat_wart_oczek_zero}
	For $n = (n_1, n_2) \in \mathbb{Z}^2$, $A_n \cap \widetilde{R}_N = \emptyset$, $x \in R_n$, $l \in \mathbb{N}$, $l \geq 2$ and $0 < t$ we have
	\begin{equation*}
	E^x \left( W(n, 0, n, l, t), e^{-\frac{1}{2} \int_0^{\sigma_l} V(X_s) \, ds} \right) \leq \frac{2^{l} 3^{3l+\alpha} C_1 C_2^{l-2} 
	C_5^l C_6(N)^{l-2} C_7^2}{q(|n|-3) (C_8 q(\min(|n_1|,|n_2|)) + 1)}.
	\end{equation*}
\end{lem}
\begin{proof}
	Recall $\mathcal{P}_{2,n} = \{ p \in \mathcal{P} : f(p) \geq f(n) \}$. By Lemma \ref{lemat_wart_oczek_S} we get
	\begin{eqnarray*}
	&& E^x \Big( W(n, 0, n, l, t), e^{-\frac{1}{2} \int_0^{\sigma_l} V(X_s) \, ds} \Big)\\
	&\leq& \sum_{p \in \mathcal{P}_{2,n}} E^x \left( S(p, l, t), e^{-\frac{1}{2} \int_0^{\sigma_l} V(X_s) \, ds} \right) \\
	&\leq& \sum_{p \in \mathcal{P}_{2,n}} \frac{2^{l} 3^{3l+\alpha} C_1 C_2^{l-2} C_5^l C_6(N)^{l-2}}{(|n_1 - p_1| + 1)^{1+\alpha} (|n_2 - p_2| + 1)^{1+\alpha} q(|n|-3) (C_8 q(\min(|n_1|,|n_2|)) + 1)} \\
	&\leq& \frac{2^{l} 3^{3l+\alpha} C_1 C_2^{l-2} C_5^l C_6(N)^{l-2} C_7^2}{q(|n|-3) (C_8 q(\min(|n_1|,|n_2|)) + 1)}.
	\end{eqnarray*}
\end{proof}

\begin{lem}\label{lemat_Wn0n1t}
	For $n = (n_1, n_2) \in \mathbb{Z}^2$, $A_n \cap \widetilde{R}_N = \emptyset$, $x \in R_n$, $k = n$ and $0 < t$ we have
	\begin{equation}
	\label{Wn0n1t}
	E^x \left( W(n, 0, k, 1, t), e^{-\frac{1}{2} \int_0^{\sigma_1} V(X_s) \, ds} \right) 
	\leq \frac{c}{(|n_1 - k_1| + 1)^{1+\alpha} (|n_2 - k_2| + 1)^{1+\alpha} q(|n| - 3)}.
	\end{equation}
\end{lem}
\begin{proof}
Observe that the left hand side of (\ref{Wn0n1t}) is bounded from above by
\begin{equation}
\label{Wn0n1ta}
\sum_{p \in \mathcal{P}} E^x \left(S(p, 1, t), e^{-\frac{1}{2} \int_0^{\sigma_{1}} V(X_s) \, ds}\right)
\end{equation}
	
Note that the above expected values are not zero only if $p$ belongs to one of the 6 sets $\mathcal{P}_1^n = \{p = (p_1, p_2) \in \mathcal{P} : p_1 = n_1 -1, |p_2 - n_2| > 1 \}$, $\mathcal{P}_2^n = \{p = (p_1, p_2) \in \mathcal{P} : p_1 = n_1, |p_2 - n_2| > 1 \}$, $\mathcal{P}_3^n = \{p = (p_1, p_2) \in \mathcal{P} : p_1 = n_1 +1, |p_2 - n_2| > 1 \}$, $\mathcal{P}_4^n = \{p = (p_1, p_2) \in \mathcal{P} : p_2 = n_2 -1, |p_1 - n_1| > 1 \}$, $\mathcal{P}_5^n = \{p = (p_1, p_2) \in \mathcal{P} : p_2 = n_2, |p_1 - n_1| > 1 \}$, $\mathcal{P}_{6}^n = \{p = (p_1, p_2) \in \mathcal{P} : n_2 = n_2 +1, |p_1 - n_1| > 1 \}$.

So (\ref{Wn0n1ta}) is equal to
$$
\sum_{i =1}^6 \sum_{p \in \mathcal{P}_i^n} E^x \left(S(p, 1, t), e^{-\frac{1}{2} \int_0^{\sigma_{1}} V(X_s) \, ds}\right).
$$

By Lemma \ref{lemat_wart_oczek_l_1} we get 
$$
 \sum_{p \in \mathcal{P}_1^n} E^x \left(S(p, 1, t), e^{-\frac{1}{2} \int_0^{\sigma_{1}} V(X_s) \, ds}\right) \le \frac{C_5 C_7}{q(|n| - 3)}.
$$
Using this, the fact that $n = k$ and similar estimates for $\mathcal{P}_2^n$, $\mathcal{P}_3^n$, $\mathcal{P}_4^n$, $\mathcal{P}_5^n$, $\mathcal{P}_6^n$ we obtain that (\ref{Wn0n1ta}) is bounded from above by
$$
\frac{6 C_5 C_7}{q(|n| - 3)} = \frac{6 C_5 C_7}{(|n_1 - k_1| + 1)^{1+\alpha} (|n_2 - k_2| + 1)^{1+\alpha} q(|n| - 3)}.
$$
\end{proof}

\begin{lem}
\label{lemat_Wn1k1t}
	For $n = (n_1, n_2) \in \mathbb{Z}^2$, $A_n \cap \widetilde{R}_N = \emptyset$, $k = (k_1, k_2) \in \mathcal{P}$, $x \in R_n$, $0 < t$ we have
	\begin{equation*}
	E^x \left( W(n, 1, k, 1, t), e^{-\frac{1}{2} \int_0^{\sigma_1} V(X_s)\, ds} \right) 
	\leq \frac{c}{(|n_1 - k_1| + 1)^{1+\alpha} (|n_2 - k_2| + 1)^{1+\alpha} q(|n| - 3)}.
	\end{equation*}
\end{lem}

\begin{proof}
	Note that the left side of the inequality is not zero only if $k_1 \in \{n_1 -1, n_1, n_1 + 1\}$ and $|k_2 - n_2| \ge 2$ or if $k_2 \in \{n_2 -1, n_2, n_2 + 1\}$ and $|k_1 - n_1| \ge 2$. Note also that
$$
E^x \left( W(n, 1, k, 1, t), e^{-\frac{1}{2} \int_0^{\sigma_1} V(X_s)\, ds} \right) \le
E^x \left( S(k,1,t), e^{-\frac{1}{2} \int_0^{\sigma_1} V(X_s)\, ds} \right).
$$
Now, the assertion of the lemma follows from Lemma \ref{lemat_wart_oczek_l_1}.
\end{proof}

Now we prove the main result of this section.

\begin{lem}\label{lemat_z_gory}
	Let $x= (x_1, x_2) \in \mathbb{R}^2$ and $0 < t$. Then
	\begin{equation*}
	T_t \mathds{1}_{\mathbb{R}^2} (x) \leq \frac{C_9 (t, N)}{(|x_1| + 1)^{1+\alpha} (|x_2| + 1)^{1+\alpha} 
	(q(\min (|x_1|, |x_2|)) + 1) (q(|x|) + 1)}.
	\end{equation*}
\end{lem}

\begin{proof}
	First assume $x= (x_1, x_2) \in \mathbb{R}^2$ such that $\max(|x_1|, |x_2|) > N+5$. We will consider small $x$ later.
	
	Let $n = (n_1, n_2) \in \mathbb{Z}^2$ such that $x \in R_n$. We have
	\begin{eqnarray}
	\nonumber
	T_t \mathds{1}_{\mathbb{R}^2} (x) &=& E^x \left( e^{-\int_0^t V(X_s) \, ds} \right) \\
	\label{Tt1}
	&\leq&  \sum_{k \in \mathcal{P}} \sum_{l = 1}^{\infty} \sum_{m=0}^l E^x\left( W(n, m, k, l, t), \sigma_{l+1} > t, e^{-\int_0^t V(X_s) \, ds} \right) \\
	\label{Tt2}
	&& + \sum_{l=1}^{\infty} E^x \left( \widetilde{S} (N+2, l, t), e^{-\int_0^t V(X_s) \, ds} \right) \\
	\label{Tt3}
	&& + E^x \left( \sigma_1 > t, e^{-\int_0^t V(X_s) \, ds} \right). 
	\end{eqnarray}
	
Note that in (\ref{Tt1}) inside the expected value the following estimates hold
\begin{eqnarray}
\nonumber
e^{-\int_0^t V(X_s) \, ds} &\leq& e^{-\frac{1}{2}\int_0^t q(|k|-3) \, ds} e^{-\frac{1}{2}\int_0^{\sigma_l} V(X_s) \, ds} \\
\nonumber
&\leq& e^{-\frac{1}{2}\int_0^t \left(\frac{1}{C_0^3} q(|k|) - \frac{1}{C_0^2} -\frac{1}{C_0} -1 \right)\, ds} e^{-\frac{1}{2}\int_0^{\sigma_l} V(X_s) \, ds}\\
\label{q3}
&=& c_0(t) e^{-\frac{t}{2C_0^3} q(|k|)} e^{-\frac{1}{2}\int_0^{\sigma_l} V(X_s) \, ds}\\
\label{q4}
&\le& c_0(t) e^{-\frac{t}{4C_0^3} q(|k|)} \frac{1}{\frac{t}{4C_0^3} q(|k|)} e^{-\frac{1}{2}\int_0^{\sigma_l} V(X_s) \, ds}.
\end{eqnarray}
Here we used (\ref{qan}).
	
Hence (\ref{Tt1}) is bounded from above by
	\begin{eqnarray}
	\label{Tt1a}
 	&&\sum_{l=2}^{\infty} c_0(t) e^{-\frac{t}{2C_0^3} q(|n|)} E^x\left( W(n, 0, n, l, t), e^{-\frac{1}{2} \int_0^{\sigma_l} V(X_s) \, ds} \right) \\
	\label{Tt1b}
	&+& \sum_{k \in \mathcal{P}} c_0(t)e^{-\frac{t}{4C_0^3} q(|k|)} \frac{1}{\frac{t}{4C_0^3} q(|k|)} \sum_{m=0}^1 E^x\left( W(n, m, k, 1, t), e^{-\frac{1}{2}\int_0^{\sigma_1} V(X_s) \, ds} \right) \quad \quad\\
	\label{Tt1c}
	&+& \sum_{k \in \mathcal{P}} \sum_{l = 2}^{\infty} \sum_{m=1}^l c_0(t)e^{-\frac{t}{2C_0^3} q(|k|)} E^x\left( W(n, m, k, l, t), e^{-\frac{1}{2} \int_0^{\sigma_l} V(X_s) \, ds} \right).
	\end{eqnarray}
	
Let us choose $N$ large enough so that 
$$
C_6(N) \leq \frac{1}{2^2 3^3 C_2 C_5 C_7}.
$$ 
Recall that $C_7 \ge 1$. By Lemma \ref{lemat_wart_oczek_zero} we obtain that (\ref{Tt1a}) is bounded from above by 
	\begin{eqnarray}
	\nonumber
	&& \sum_{l=2}^{\infty} c_0(t)e^{-\frac{t}{2C_0^3} q(|n|)} \frac{2^{l} 3^{3l+\alpha} C_1 C_2^{l-2} C_5^l C_6(N)^{l-2} C_7^2}
	{q(|n|-3) (C_8 q(\min(|n_1|,|n_2|)) + 1)} \\
	\nonumber
	&\leq& \sum_{l=2}^{\infty} c_0(t)e^{-\frac{t}{4C_0^3} q(|n|)} \frac{1}{2^l}  
	\frac{2^4 3^{\alpha+6} C_1 C_5^2 C_7^{2}}{q(|n|-3) (C_8 q(\min(|n_1|,|n_2|)) + 1)} \\
	\label{Tt1aa}
	&=& c_0(t)e^{-\frac{t}{4C_0^3} q(|n|)}   \frac{c_1}{q(|n|-3) (C_8 q(\min(|n_1|,|n_2|)) + 1)}.
	\end{eqnarray}
	
$q$ satisfies \eqref{warunek_logarytmu}, so we may choose $N$ large enough, so that for any $k \in \mathcal{P}$ we have 
$q(|k|) \ge \frac{1}{C_0} q(|k|+1) -1 \geq \frac{8(1+\alpha)C_0^3}{t} \ln (|k|+1) - 1 = \frac{4(1+\alpha)C_0^3}{t} \ln ((|k|+1)^2) - 1
\geq \frac{4(1+\alpha)C_0^3}{t} \ln ((|k_1| + 1) \, (|k_2| + 1))$. Then for any $k \in \mathcal{P}$ we have
	\begin{eqnarray}
	\nonumber
	e^{-\frac{t}{4C_0^3} q(|k|)} &\leq& e^{-\frac{t}{4C_0^3} \frac{4(1+\alpha)C_0^3}{t} \ln((|k_1| + 1) (|k_2| + 1))}\\
	\nonumber
	&=& 
	e^{-(1+\alpha) \ln((|k_1| + 1) (|k_2| + 1))}\\ 
	\label{exp_estimate}
	&=& \frac{1}{(|k_1| + 1)^{1+\alpha} (|k_2| + 1)^{1+\alpha}}.
	\end{eqnarray}
	
	Note that $n \in \mathcal{P}$, so using (\ref{exp_estimate}) we obtain that (\ref{Tt1aa}) is bounded from above by
	\begin{equation*}
\frac{c_2(t)}{(|n_1|+1)^{1+\alpha} (|n_2|+1)^{1+\alpha} q(|n| - 3) (C_8 q(\min(|n_1|,|n_2|)) + 1)}.
	\end{equation*}
	
	Now we estimate (\ref{Tt1b}).
	
	$E^x\left( W(n, m, k, 1, t), e^{-\frac{1}{2}\int_0^{\sigma_1} V(X_s) \, ds} \right)$ is not zero only if $k_1 \in \{n_1 - 1, n_1, n_1 + 1\}$ or $k_2 \in \{n_2 - 1, n_2, n_2 + 1\}$. Then $q(|k|) \geq C_8 q(\min(|n_1|,|n_2|)) + 1$, so 
\begin{equation}
\label{1V_estimate}
\frac{1}{\frac{t}{4C_0^3} q(|k|)} \leq \frac{c_3(t)}{C_8 q(\min(|n_1|,|n_2|)) + 1}.
\end{equation}
Moreover, $E^x\left( W(n, 0, k, 1, t), e^{-\frac{1}{2}\int_0^{\sigma_1} V(X_s) \, ds} \right)$ is not zero only if $k = n$. Similarly,  
$E^x\left( W(n, 1, k, 1, t), e^{-\frac{1}{2}\int_0^{\sigma_1} V(X_s) \, ds} \right)$ is not zero only if $k_1 \in \{n_1 -1, n_1, n_1 + 1\}$ and $|k_2 - n_2| \ge 2$ or if $k_2 \in \{n_2 -1, n_2, n_2 + 1\}$ and $|k_1 - n_1| \ge 2$.
	
By Lemmas \ref{lemat_Wn0n1t}, \ref{lemat_Wn1k1t} and by (\ref{exp_estimate}), (\ref{1V_estimate}) we obtain that (\ref{Tt1b}) is bounded from above by
\begin{eqnarray*}
&& \sum_{k \in \mathcal{P}} 
	 \frac{c_4(t) }{(|k_1|+1)^{1+\alpha}(|k_2|+1)^{1+\alpha}(|n_1 - k_1| + 1)^{1+\alpha} (|n_2 - k_2| + 1)^{1+\alpha}} \\
&&	\quad \quad \times \frac{1}{q(|n| - 3) (C_8 q(\min(|n_1|,|n_2|)) + 1)}.
\end{eqnarray*}
By Lemma \ref{auxiliary_series} this is smaller than
$$
\frac{c_5(t)}{(|n_1|+1)^{1+\alpha} (|n_2|+1)^{1+\alpha} q(|n| - 3) (C_8 q(\min(|n_1|,|n_2|)) + 1)},
$$
which finishes estimates of (\ref{Tt1b}).
	
By Corollary \ref{wniosek_szacowanie_W} and the fact that $C_2, C_7 \ge 1$ we obtain that (\ref{Tt1c}) is bounded from above by
	\begin{equation*}
	\begin{split}
	&\sum_{k \in \mathcal{P}} \sum_{l = 2}^{\infty} \sum_{m=1}^l  \frac{c_0(t)e^{-\frac{t}{2C_0^3} q(|k|)} 2^{l} 3^{3l+\alpha} C_1 C_2^{l-1} C_5^l C_6(N)^{l-2} C_7^l}{(|n_1 - k_1| + 1)^{1+\alpha} (|n_2 - k_2| + 1)^{1+\alpha} q(|n|-3) (C_8 q(\min(|n_1|,|n_2|)) + 1)} \\
	&= \sum_{k \in \mathcal{P}} \sum_{l = 2}^{\infty} l \, \frac{c_0(t)e^{-\frac{t}{2C_0^3} q(|k|)} 2^{l} 3^{3l+\alpha} C_1 C_2^{l-1} C_5^l C_6(N)^{l-2} C_7^l}{(|n_1 - k_1| + 1)^{1+\alpha} (|n_2 - k_2| + 1)^{1+\alpha} q(|n|-3) (C_8 q(\min(|n_1|,|n_2|)) + 1)}.
	\end{split}
	\end{equation*}
	Recall $C_6(N) \leq \frac{1}{2^2 3^3 C_2 C_5 C_7}$. Using this and (\ref{exp_estimate}) we get that the expression above is smaller or equal to
	\begin{eqnarray*}
	&&\sum_{k \in \mathcal{P}} \sum_{l = 2}^{\infty} \frac{l}{2^l} \, \frac{c_0(t)e^{-\frac{t}{2C_0^3} q(|k|)} \, 2^4 \, 3^{\alpha+6} C_1 C_2 C_5^2 C_7^2}{(|n_1 - k_1| + 1)^{1+\alpha} (|n_2 - k_2| + 1)^{1+\alpha} q(|n|-3) (C_8 q(\min(|n_1|,|n_2|)) + 1)} \\
	&&\leq \sum_{k \in \mathcal{P}} 
	\frac{c_0(t) \, c_6}{(|k_1|+1)^{1+\alpha}(|k_2|+1)^{1+\alpha}(|n_1 - k_1| + 1)^{1+\alpha} (|n_2 - k_2| + 1)^{1+\alpha}} \\
	&&	\quad \quad \times \frac{1}{q(|n| - 3) (C_8 q(\min(|n_1|,|n_2|)) + 1)}.
	\end{eqnarray*}
	By Lemma \ref{auxiliary_series} this is smaller than
	\begin{equation*}
	\frac{c_7 (t)}{(|n_1| + 1)^{1 + \alpha} (|n_2| + 1)^{1 + \alpha} q(|n|-3) (C_8 q(\min(|n_1|,|n_2|)) + 1)}.
	\end{equation*}
	
	Thus we obtain that (\ref{Tt1}) is bounded from above by
	\begin{equation}\label{nier_pomoc_2}
	\frac{c_{8}(t)}{(|n_1| + 1)^{1+\alpha} (|n_2| + 1)^{1+\alpha} q(|n|-3) (C_8 q(\min(|n_1|,|n_2|)) + 1)}.
	\end{equation}
	
	Now, let us observe that (\ref{Tt2}) is smaller or equal to
	\begin{equation*}
	\sum_{l=2}^{\infty} E^x \left( \widetilde{S} (N+2, l, t), e^{-\frac{1}{2}\int_0^{\sigma_l} V(X_s) \, ds} \right)
	+ E^x \left( \widetilde{S} (N+2, 1, t), e^{-\frac{1}{2}\int_0^{\sigma_1} V(X_s) \, ds} \right)
	\end{equation*}
	
	From Lemma \ref{wniosek_S_z_fala} and the facts that $C_6(N) \leq \frac{1}{2^2 3^3 C_2 C_5 C_7}$  and $C_7 \ge 1$ we get
	\begin{eqnarray}
	\nonumber
	&&\sum_{l=2}^{\infty} E^x \left( \widetilde{S} (N+2, l, t), e^{-\frac{1}{2}\int_0^{\sigma_l} V(X_s) \, ds} \right)\\
	\nonumber
	&\leq& \sum_{l=2}^{\infty} \frac{2^{l} 3^{3l+\alpha} C_1 C_2^{l-2} C_5^{l} C_6(N)^{l-2} \widetilde{C}_5(N)}{(| |n_1| - N| + 1)^{1+\alpha} (| |n_2| - N| + 1)^{1+\alpha} q(|n| -3) (C_8 q(\min(|n_1|, |n_2|)) + 1)} \\
	\nonumber
	&\leq& \sum_{l=2}^{\infty} \frac{1}{2^l}   \frac{2^4 \, 3^{\alpha+6} C_1 C_5^2 C_7^{-l+2} \widetilde{C}_5 (N)}{(||n_1| - N| + 1)^{1+\alpha} (| |n_2| - N| + 1)^{1+\alpha} q(|n| -3) (C_8 q(\min(|n_1|,|n_2|)) + 1)} \\
	\label{S_tilde}
	&\leq& \frac{c_{9}(N)}{(| |n_1| - N| + 1)^{1+\alpha} (| |n_2| - N| + 1)^{1+\alpha} q(|n| -3) (C_8 q(\min(|n_1|,|n_2|)) + 1)}.
	\end{eqnarray}
	Note that for $|n_1| \geq 2N$
	\begin{equation}
	\label{n1N_large}
	(| |n_1| - N| + 1)^{1+\alpha} \geq \left(\frac{1}{2} |n_1| + 1\right)^{1+\alpha} \geq \frac{1}{2^{1+\alpha}} (|n_1| + 1)^{1+\alpha}
	\end{equation}
	and for $|n_1| < 2N$
	\begin{equation}
	\label{n1N_small}
	(| |n_1| - N| + 1)^{1+\alpha} \geq 1 \geq \frac{1}{(2N + 1)^{1+\alpha}} (|n_1| + 1)^{1+\alpha}.
	\end{equation}
Similar estimates hold for $(| |n_2| - N| + 1)^{1+\alpha}$. Hence, (\ref{S_tilde}) is bounded from above by
\begin{equation*}
	\frac{c_{10}(N)}{(|n_1| + 1)^{1+\alpha} (|n_2| + 1)^{1+\alpha} q(|n|-3) (C_8 q(\min(|n_1|,|n_2|)) + 1)}.
	\end{equation*}

	Now, let us estimate $E^x \left( \widetilde{S} (N+2, 1, t), e^{-\frac{1}{2}\int_0^{\sigma_1} V(X_s) \, ds} \right)$. 
	There are two cases: $-N-2 \leq n_1 \leq N+3$ or $-N-2 \leq n_2 \leq N+3$, otherwise the set $\widetilde{S} (N+2, 1, t)$ is empty. Let us assume that $-N-2 \leq n_1 \leq N+3$ (the estimates for the case $-N-2 \leq n_2 \leq N+3$ are similar). From Lemma \ref{wniosek_S_z_fala}, we get
	\begin{eqnarray}
	\nonumber
	&&E^x \Big( \widetilde{S} (N+2, 1, t), e^{-\frac{1}{2}\int_0^{\sigma_1} V(X_s) \, ds} \Big)\\
	\nonumber
	&\leq& \frac{\tilde{C}_5(N)}{(|n_2| -N)^{1+\alpha} q(|n|-3)} \\
	\label{Tt2b}
	&\leq& \frac{c_{11}(N)}{(|n_1| + 1)^{1+\alpha} (| |n_2| -N| + 1)^{1+\alpha} q(|n|-3)},
	\end{eqnarray}
	because $-N-2 \leq n_1 \leq N+3$. 
	
	Using (\ref{n1N_large}) and (\ref{n1N_small}) (for $n_2$ instead of $n_1$) and the fact that $C_8 q(\min(|n_1|,|n_2|)) + 1 \leq c(N)$ we obtain that (\ref{Tt2b}) is smaller or equal to
	\begin{equation*}
	\frac{c_{12}(N)}{(|n_1| + 1)^{1+\alpha} (|n_2| + 1)^{1+\alpha} q(|n|-3) (C_8 q(\min(|n_1|,|n_2|)) + 1)}.
	\end{equation*}
	
	Hence (\ref{Tt2}) is bounded from above by
	\begin{equation}
	\label{nier_pomoc_3}
	\frac{c_{13}(N)}{(|n_1| + 1)^{1+\alpha} (|n_2| + 1)^{1+\alpha} q(|n|-3) (C_8 q(\min(|n_1|,|n_2|)) + 1)}.
	\end{equation}
	
	Now let us estimate (\ref{Tt3}). We have
	\begin{equation*}
	E^x \left( \sigma_1 > t, e^{-\int_0^t V(X_s) \, ds} \right)  \leq E^x \left( \sigma_1 > t, e^{-\int_0^t q(|n| - 3) \, ds} \right) 
	\le e^{-t q(|n| - 3)}.
	\end{equation*}
	Note that $n \in \mathcal{P}$. By (\ref{q3}) and (\ref{exp_estimate}) we get
	\begin{equation*}
	e^{-\frac{t}{2} q(|n| -3)} \leq \frac{c_{14}(t)}{(|n_1| + 1)^{1+\alpha} (|n_2| + 1)^{1+\alpha}}.
	\end{equation*}
	By arguments similar to (\ref{qpn_estimate}) we obtain
	\begin{equation*}
	e^{-\frac{t}{2} q(|n| - 3)} \leq \frac{2}{\frac{t^2}{2^2} (q(|n| - 3))^2} \leq 
	\frac{2}{\frac{t^2}{2^2} q(|n| - 3) (C_8 q(\min(|n_1|,|n_2|)) + 1)}.
	\end{equation*}
	Thus (\ref{Tt3}) is bounded from above by
	\begin{equation}
	\label{nier_pomoc_4}
	\frac{c_{15}(t)}{(|n_1| + 1)^{1+\alpha} (|n_2| + 1)^{1+\alpha} q(|n| - 3) (C_8 q(\min(|n_1|,|n_2|)) + 1)}.
	\end{equation}
	
	Using \eqref{nier_pomoc_2}, \eqref{nier_pomoc_3} and \eqref{nier_pomoc_4}, we get
	\begin{equation*}
	T_t \mathds{1}_{\mathbb{R}^2} (x) \leq \frac{c_{16}(t, N)}{(|n_1| + 1)^{1+\alpha} (|n_2| + 1)^{1+\alpha} 
	q(|n| - 3) (C_8 q(\min(|n_1|,|n_2|)) + 1)}.
	\end{equation*}
	
	Recall $n_1 - 1 < x_1 \le n_1$ and $n_2 - 1 < x_2 \le n_2$. Therefore
	\begin{equation*}
	\frac{1}{(|n_1| + 1)^{1+\alpha} (|n_2| + 1)^{1+\alpha}} \leq \frac{c_{17}}{(|x_1| + 1)^{1+\alpha} (|x_2| + 1)^{1+\alpha}}.
	\end{equation*}
	By (\ref{qan}) we obtain
	\begin{equation*}
	q(|n|-3) \geq q(|x| - 4) \geq c_{18} q(|x|) - c_{19},
	\end{equation*}
	where $c_{18} = C_0^{-4}$ and $c_{19} = 1 + C_0^{-1} + C_0^{-2}+ C_0^{-3}$. 
	
	Recall that we assumed $\max(|x_1|, |x_2|) > N+5$. Let us choose $N$ large enough so that for $x \in \R^2$ such that $\max(|x_1|, |x_2|) > N+5$ we have $q(|x|) - \frac{2c_{19}}{c_{18}} \geq 1$. Then we have 
	\begin{equation*}
	c_{18} q(|x|) - c_{19}
	= \frac{c_{18}}{2} \left(q(|x|) + q(|x|) - \frac{2c_{19}}{c_{18}}\right) \geq \frac{c_{18}}{2} (q(|x|) + 1).
	\end{equation*}
	 Similarly
	\begin{equation*}
	C_8 q(\min(|n_1|,|n_2|)) + 1 \geq c_{20} (q(\min(|x_1|, |x_2| )) + 1).
	\end{equation*}
	
	Therefore
	\begin{equation*}
	T_t \mathds{1}_{\mathbb{R}^2} (x) \leq \frac{c_{21}(t, N)}{(|x_1| + 1)^{1+\alpha} (|x_2| + 1)^{1+\alpha} 
	(q(|x|) + 1) (q(\min(|x_1|, |x_2| )) + 1)}.
	\end{equation*}
	
	Now let us take $x =(x_1, x_2)$, such that $\max (|x_1|, |x_2|) \leq N+5$. Then
	\begin{equation*}
	T_t \mathds{1}_{\mathbb{R}^2} (x) = E^x \left( e^{-\int_0^t V(X_s) \, ds} \right) \leq 1.
	\end{equation*}
	Note that this is smaller or equal to
	$$
	\frac{c_{22}(N)}{(|x_1| + 1)^{1+\alpha} (|x_2| + 1)^{1+\alpha} (q(|x|) + 1) (q(\min(|x_1|, |x_2|)) + 1)},
	$$
	because in this case
	\begin{eqnarray*}
	&&(|x_1| + 1)^{1+\alpha} (|x_2| + 1)^{1+\alpha} (q(|x|) + 1) (q(\min\{|x_1|, |x_2| \}) + 1) \\
	&\leq& (N + 5 + 1)^{1+\alpha} (N + 5 + 1)^{1+\alpha} (q(N+5) + 1) (q(N+5) + 1).
	\end{eqnarray*}
	This finishes the proof of Lemma \ref{lemat_z_gory}.
\end{proof}

\section{Proof of the main result}
The aim of this section is to present the proof of Theorem \ref{main} using results proven in Sections 3 and 4.

Denote $e_1 = (1,0)$.
\begin{lem}
\label{TtD}
Let $D = (-2,2) \times (-2,2)$, $t > 0$. For any $x \ge 3$ we have
$$
T_t \mathds{1}_D(xe_1) \le c t x^{-1-\alpha}.
$$
\end{lem}
\begin{proof}
Recall that for any $s >0$, $y = (y_1,y_2) \in \R^2$, $z = (z_1,z_2) \in \R^2$ we have $p(s,y,z) = \tilde{p}(s,y_1,z_1) \tilde{p}(s,y_2,z_2)$. It is well known (see e.g. \cite{BG1960} or \cite{BGR2010}) that for any $s > 0$, $y_1, z_1 \in \R$ we have $\tilde{p}(s,y_1,z_1) \le c_1 s |y_1 - z_1|^{-1 - \alpha}$.

For any $x \in \R$ we get
\begin{equation}
T_t \mathds{1}_D(xe_1) = E^{x e_1} \left( X_t \in D, e^{-\int_0^t V(X_s) \, ds} \right)
\label{pp}
\le \int_{-2}^2 \int_{-2}^2 \tilde{p}(t,x,y_1) \tilde{p}(t,0,y_2) \, dy_1 \, dy_2.
\end{equation}
Hence for $x \ge 3$ the right-hand side of (\ref{pp}) is bounded from above by
$$
\int_{-2}^2 \frac{c_1 t \, dy_1}{|x - y_1|^{1+\alpha}} \int_{-\infty}^{\infty} \tilde{p}(t,0,y_2) \, dy_2
\le \frac{4 c_1 t}{(x - 2)^{1+\alpha}}
\le \frac{3^{1+\alpha} 4 c_1 t}{x^{1+\alpha}}.
$$
\end{proof}

\begin{proof}[proof of Theorem \ref{main}] Put $D= (-2, 2) \times (-2, 2)$.
First, we show that if \eqref{warunek_logarytmu} is not satisfied, then the semigroup is not intrinsically ultracontractive. If \eqref{warunek_logarytmu} does not hold, then there exist a constant $M > 0$ and a sequence $(x_n)$ such that $\displaystyle \lim_{n \to \infty} x_n = \infty$ and for any $n \in \N$ we have 
\begin{equation*}
x_n \ge 3, \quad \quad \frac{q(x_n)}{\ln x_n} \le M.
\end{equation*}

By Assumptions (A)  for any $x \ge 0$ we have $q(x+1) \leq C_0 q(x) + C_0$, so for any $t > 0$ and $n \in \N$ we get
\begin{eqnarray}
\nonumber
T_t \mathds{1}_{B(x_n e_1,1)} (x_n e_1) &=& E^{x_n e_1} \left( X_t \in B(x_n e_1,1), e^{-\int_0^t V(X_s) \, ds} \right) \\
\nonumber
&\geq& E^{x_n e_1} \left( \tau_{B(x_n e_1, 1)} > t, e^{-\int_0^t q(x_n + 1) \, ds} \right)\\
\nonumber
&\geq& P^{x_n e_1} \left(\tau_{B(x_n e_1,1)} > t \right) e^{-t C_0} e^{-t C_0 q(x_n)} \\
\nonumber
&\geq& P^0 \left( \tau_{B(0,1)} >t \right) e^{-t C_0} e^{-t C_0 M \ln x_n}\\
\label{TtB}
&=& c_1(t) x_n^{-t C_0 M}.
\end{eqnarray}
Using Lemma \ref{TtD} and (\ref{TtB}) we obtain that for any $t > 0$ 
$$
\frac{T_t \mathds{1}_{B(x_n e_1,1)} (x_n e_1)}{T_t \mathds{1}_D(x_n e_1)} 
\ge \frac{c_1(t) x_n^{-t C_0 M}}{c t x_n^{-1-\alpha}}
= \frac{c_1(t)}{c t} x_n^{1 + \alpha -t C_0 M}
$$
Choose $t> 0$ small enough so that $1 + \alpha - t C_0 M > 0$. Recall that $\displaystyle \lim_{n \to \infty} x_n = \infty$. It follows that Condition \ref{warunek_IU2} is not satisfied, so the semigroup $\{T_t: \, t \ge 0 \}$ is not intrinsically ultracontractive.

Now, assume that \eqref{warunek_logarytmu} is satisfied. Fix $t > 0$. For any $x \in \R^2$ by Lemmas \ref{lemat_z_dolu} and \ref{lemat_z_gory} we obtain
\begin{eqnarray*}
T_t \mathds{1}_{\mathbb{R}^2} (x) &\leq& \frac{C_{9}(t, N)}{(|x_1| + 1)^{1+\alpha} (|x_2| + 1)^{1+\alpha} 
(q(\min (|x_1|, |x_2|)) + 1) (q(|x|) + 1)} \\
&\leq& \frac{C_{9}(t, N)}{C_3(t)} T_t \mathds{1}_D (x).
\end{eqnarray*}
Thus Condition \ref{warunek_IU} is satisfied. So, the semigroup $\{T_t: \, t \ge 0\}$ is intrinsically ultracontractive.

Now let us estimate the first eigenfunction. Recall that $t > 0$ is fixed. Since $\varphi_1$ is bounded, for any $x \in \R^2$ we get $\varphi_1 (x) = e^{\lambda_1 t} T_t \varphi_1 (x) \leq e^{\lambda_1 t} T_t (\|\varphi_1\|_{\infty} \mathds{1}_{\mathbb{R}^2}) (x)$.  By Lemma \ref{lemat_z_gory} we obtain for any $x \in \R^2$
\begin{equation*}
\varphi_1 (x) \le 
\frac{\|\varphi_1\|_{\infty} e^{\lambda_1 t} C_{9}(t,N)}{(|x_1| + 1)^{1+\alpha} (|x_2| + 1)^{1+\alpha} 
(q(\min(|x_1|, |x_2|)) + 1) (q(|x|) + 1)}.
\end{equation*}

Lower bound estimates are similar. The function $\varphi_1$ is continuous and positive on $\overline{D}$, so there exists a constant $c_2 > 0$, such that $\varphi_1 (x) > c_2$ for any $x \in \overline{D}$. Hence, for any $x \in \R^2$ we get $\varphi_1 (x) = e^{\lambda_1 t} T_t \varphi_1 (x) \geq e^{\lambda_1 t} T_t (c_2 \mathds{1}_D) (x)$. From Lemma \ref{lemat_z_dolu}, we obtain for any $x \in \R^2$
$$
\varphi_1 (x) 
\geq \frac{c_2 e^{\lambda_1 t} C_3(t)}{(|x_1| + 1)^{1+\alpha} (|x_2| + 1)^{1+\alpha} (q(\min(|x_1|, |x_2|)) + 1) (q(|x|) + 1)}.
$$
\end{proof}

\end{document}